\documentclass[12pt]{amsart}
\usepackage{amssymb}
\usepackage[all]{xypic}
\usepackage[all]{xy}
\usepackage{moreverb}

\newcounter{alphalistcntr}

\theoremstyle{plain}
\newtheorem{theorem}{Theorem}[section]
\newtheorem{lemma}[theorem]{Lemma}

\newtheorem{prop}[theorem]{Proposition}

\newtheorem{algorithm}[theorem]{Algorithm}

\theoremstyle{remark}

\newtheorem{remark}[theorem]{Remark}

\newtheorem*{note*}{Note}
\newtheorem*{remark*}{Remark}
\newtheorem*{example*}{Example}

\theoremstyle{definition}
\newtheorem*{definition*}{Definition}

\newcommand{\Z}{\mathbb{Z}}
\newcommand{\R}{\mathbb{R}}
\newcommand{\Q}{\mathbb{Q}}

\newcommand{\N}{\mathbb{N}}

\newcommand{\Image}{\mathrm{Img}}
\newcommand{\Gal}{\mathrm{Gal}}

\newcommand{\op}{\mathrm{op}}

\newcommand{\calO}{\mathcal{O}}
\newcommand{\calM}{\mathcal{M}}
\newcommand{\calT}{\mathcal{T}}

\newcommand{\fra}{\mathfrak{a}}

\newcommand{\frc}{\mathfrak{c}}

\newcommand{\frg}{\mathfrak{g}}
\newcommand{\frp}{\mathfrak{p}}

\newcommand{\coker}{\mathrm{coker}}

\newcommand{\ann}{\mathrm{ann}}
\newcommand{\nr}{\mathrm{nr}}

\title[Computing generators of free modules over orders in group algebras II]{Computing generators of free modules \\ over orders in group algebras II}

\author{Werner Bley}
\address{Werner Bley\\
Fachbereich f\"ur Mathematik und Informatik der Universit\"at Kassel\\
Heinrich-Plett-Str.\ 40\\
34132 Kassel\\
Germany}
\email{bley@mathematik.uni-kassel.de}
\urladdr{http://www.mathematik.uni-kassel.de/$\sim$bley}

\author{Henri Johnston}
\address{Henri Johnston\\ 
St.\ John's College\\
Cambridge CB2 1TP\\
United Kingdom
}
\email{H.Johnston@dpmms.cam.ac.uk}
\urladdr{http://www.dpmms.cam.ac.uk/$\sim$hlj31}

\date{15th September 2010}

\begin{document}

\begin{abstract}
Let $E$ be a number field and $G$ be a finite group.
Let $\mathcal{A}$ be any $\mathcal{O}_{E}$-order of full rank
in the group algebra $E[G]$ and $X$ be a (left) $\mathcal{A}$-lattice.
In a previous article, we gave a necessary and sufficient condition for $X$ 
to be free of given rank $d$ over $\mathcal{A}$. In the case that (i) the Wedderburn decomposition $E[G] \cong \oplus_{\chi} M_{\chi}$ 
is explicitly computable and (ii) each $M_{\chi}$ is in fact a matrix ring over a field,
this led to an algorithm that either gives elements 
$\alpha_{1}, \ldots, \alpha_{d} \in X$ such that 
$X=\mathcal{A}\alpha_{1} \oplus \cdots \oplus \mathcal{A}\alpha_{d}$ 
or determines that no such elements exist. 
In the present article, we generalise the algorithm by weakening condition (ii) considerably.
\end{abstract}

\maketitle

\section{Introduction}\label{intro}

Let $E$ be a number field and $G$ be a finite group.
Let $\mathcal{A}$ be any $\mathcal{O}_{E}$-order of full rank in the
group algebra $E[G]$ and $X$ be a (left) $\mathcal{A}$-lattice, i.e.\
a (left) $\mathcal{A}$-module that is finitely generated and torsion-free 
over $\mathcal{O}_{E}$.
The main theoretical result of \cite{MR2422318} is a necessary and sufficient condition for 
$X$ to be free of given rank $d$ over $\mathcal{A}$. 
In order to use this criterion
for computational purposes, we had to impose two hypotheses:

\begin{itemize}
\item [(H1)] The Wedderburn decomposition $E[G] \cong \oplus_{\chi} M_{\chi}$,
where each $M_{\chi} = M_{n_\chi}(D_\chi)$ is a matrix ring over a skew field $D_\chi$, 
is explicitly computable.
\item [(H2)] The Schur indices of all $E$-rational irreducible characters of $G$ are equal to 1,
i.e.\ each $D_{\chi}$ above is in fact a number field.
\end{itemize}

Under these hypotheses, an algorithm was given that either computes elements 
$\alpha_{1}, \ldots, \alpha_{d} \in X$ such that 
$X=\mathcal{A}\alpha_{1} \oplus \cdots \oplus \mathcal{A}\alpha_{d}$ 
or determines that no such elements exist. 
In the present article, we generalise this result by retaining hypothesis (H1) but relaxing (H2)
considerably.

Before outlining the new hypothesis (H2\textprime) which replaces (H2), 
we briefly introduce some notation.
Let $D$ be a skew field that is central and finite-dimensional over a number field $F$.
Let $\nr: D \longrightarrow F$ denote the reduced norm map and let 
$\Delta \subseteq D$ be a maximal $\mathcal{O}_{F}$-order. 
Then $\nr(\Delta^{\times})  \subseteq \mathcal{O}_{F}^{\times +}$, where 
$\mathcal{O}_{F}^{\times +}$ is a certain subgroup of finite index in $\mathcal{O}_{F}^{\times}$. 

\begin{itemize}
\item[(H2\textprime)] For every Wedderburn component 
$M_{n_{\chi}}(D_{\chi})$ of $E[G]$, the following conditions hold
(we omit the $\chi$ subscripts and use the notation from above):
\begin{itemize}
\item [(a)] if $nd>1$, then $\Delta$ has the locally free cancellation property (see \S \ref{subsec:loc-free-cancel});
\item [(b)] if $nd>1$, then $\nr(\Delta^{\times}) = \mathcal{O}_{F}^{\times +}$ (see \S \ref{subsec:norm});
\item [(c)] if $\nr(\Delta^{\times}) \neq \mathcal{O}_{F}^{\times +}$ then we can compute a set of generators of $\Delta^{\times}$,\\
else we can compute a set of representatives of 
$\nr: \Delta^{\times} \longrightarrow  \mathcal{O}_{F}^{\times +}$ (see \S \ref{subsec:norm-unit-reps}); and
\item [(d)] we can solve the principal ideal problem for fractional left $\Delta$-ideals 
(see \S \ref{subsec:PIB}).
\end{itemize}
\end{itemize}

In particular, (H2\textprime) holds whenever $E=\Q$ and $|G|<32$, or (H2) holds (for example, $G$ is abelian, dihedral, symmetric on any number of letters, or nilpotent of odd order).
If we assume that $d=1$, then (H2\textprime) is satisfied whenever $E=\Q$ and $G$ is any generalised quaternion group. Furthermore, if $D_{\chi}$ is \emph{not} a totally definite quaternion algebra then (a) and (b) hold; if $D_{\chi}$ is a totally definite quaternion algebra then (c) holds; if 
$D_{\chi}$ is any quaternion algebra then (d) holds; and if $F=\Q$ then (b) holds.
We note that if the full strength of (H2\textprime) does not hold then it may still be possible to run the algorithm and find generators if they exist, though this is not guaranteed
(see Remark \ref{rmk:slightly-weaker}).
For a detailed discussion of (H2\textprime) and the conditions under which it is satisfied, we refer the reader to \S \ref{sec:hypotheses}. 

The original motivation for this work comes from the following special case.
Let $L/K$ be a finite Galois extension of number fields with Galois group $G$
such that $E$ is a subfield of $K$ and put $d=[K:E]$. One can take $X=\mathcal{O}_{L}$
and $\mathcal{A}=\mathcal{A}( E[G]; \mathcal{O}_{L}) 
:= \{ \lambda \in E[G] \mid \lambda\mathcal{O}_{L} \subseteq \mathcal{O}_{L} \}$.
The application of the algorithm to this special situation is implemented in Magma (\cite{MR1484478}) 
under certain extra hypotheses when $K=E=\Q$ (see \S \ref{sec:comp-results}). 
The source code and input files are available from
\texttt{http://www.mathematik.uni-kassel.de/$\sim$bley/pub.html}.
For further discussion of the motivating special case and a review of the relevant literature, we refer the reader to the introduction of \cite{MR2422318}.

\section{A necessary and sufficient condition for freeness}

We briefly recall (with some minor differences and corrections) relevant notation and results from \cite[\S 2]{MR2422318}. For further background material we refer the reader to \cite{MR1972204}.

Let $E$ be a number field with ring of integers $\mathcal{O}_{E}$ and let $G$ be a finite group. 
Let $\mathcal{A}$ be any $\mathcal{O}_{E}$-order in the group algebra $A:=E[G]$, and let 
$\mathcal{M}$ denote some fixed maximal $\mathcal{O}_{E}$-order in $A$ containing $\mathcal{A}$.
(In fact, the results of this section still hold if $A$ is replaced by any finite-dimensional semisimple 
$E$-algebra.)

If $\mathfrak{p}$ is a prime of $\mathcal{O}_{E}$, we write $\mathcal{O}_{E,\mathfrak{p}}$ for the
localisation (not completion) of $\mathcal{O}_{E}$ at $\mathfrak{p}$. 
More generally, if $M$ is an $\mathcal{O}_{E}$-module, we write 
$M_{\mathfrak{p}} := \mathcal{O}_{E,\mathfrak{p}} \otimes_{\mathcal{O}_{E}} M$
for the localisation of $M$ at $\mathfrak{p}$. 
Let $X$ be a left $\mathcal{A}$-lattice, i.e.\ a left $\mathcal{A}$-module that is finitely generated and torsion-free over $\mathcal{O}_{E}$. Then we say that $X$ is locally free of rank $d \in \N$ if for every prime
$\mathfrak{p}$ of $\mathcal{O}_{E}$, we have $X_{\mathfrak{p}}$ free of rank $d$ over $\mathcal{A}_{\mathfrak{p}}$. 
We set 
$
\mathcal{M}X := \{ \sum_{i=1}^{r} \lambda_{i} x_{i} \mid \lambda_{i} \in \mathcal{M}, x_{i} \in X, r \in \N \}
$,
which is an $\mathcal{O}_{E}$-submodule of the $E$-vector space $E \otimes_{\mathcal{O}_{E}} X$.
If $X$ is locally free over $\mathcal{A}$ then we can (and often do) identify $\mathcal{M}X$ with 
$\mathcal{M} \otimes_{\mathcal{A}} X$.

Let $e_{1}, \ldots, e_{r}$ denote the primitive central idempotents of $A$. Setting $A_{i}:=Ae_{i}$
and $\mathcal{M}_{i} := \mathcal{M}e_{i}$, we have decompositions
\[
A=A_{1} \oplus \cdots \oplus A_{r} \quad \textrm{ and } \quad 
\mathcal{M} = \mathcal{M}_{1} \oplus \cdots \oplus \mathcal{M}_{r}.
\]
Let $\mathfrak{f}$ be any full two-sided ideal of $\mathcal{M}$ contained in $\mathcal{A}$.
Then we have
$\mathfrak{f} \subseteq \mathcal{A} \subseteq \mathcal{M} \subseteq A$. 
Set $\overline{\mathcal{M}} := \mathcal{M}/\mathfrak{f}$ and 
$\overline{\mathcal{A}} := \mathcal{A}/\mathfrak{f}$ so that 
$\overline{\mathcal{A}} \subseteq \overline{\mathcal{M}}$ are finite rings, and denote the canonical
map $\mathcal{M} \longrightarrow \overline{\mathcal{M}}$ by $m \mapsto \overline{m}$.
Note that we have decompositions
\[
\mathfrak{f} = \mathfrak{f}_{1} \oplus \cdots \oplus \mathfrak{f}_{r} 
\quad  \textrm{ and } \quad \overline{\mathcal{M}} =  
\overline{\mathcal{M}_{1}} \oplus \cdots \oplus \overline{\mathcal{M}_{r}},
\]
where each $\mathfrak{f}_{i}$ is a non-zero ideal of $\mathcal{M}_{i}$ and 
$\overline{\mathcal{M}_{i}} := \mathcal{M}_{i} / \mathfrak{f}_{i}$.

Now fix $d \in \N$, and for the rest of this section suppose $1 \leq i \leq r$ and $1 \leq j \leq d$.
(We shall now abuse notation slightly by not distinguishing between a noncommutative ring $R$ and its opposite ring $R^{\op}$ since they are equal as sets - see \cite[top of p.839]{MR2422318}.)
For each $i$, let $U_{i} \subset GL_{d}(\mathcal{M}_{i})^{}$ denote a set of representatives of the image of the
natural projection $GL_{d}(\mathcal{M}_{i})^{} \longrightarrow GL_{d}(\overline{\mathcal{M}_{i}})^{}$.
We now recall without proof \cite[Corollary 2.4]{MR2422318}, which is the key theoretical result leading to Algorithm \ref{alg:the-alg}. 

\begin{theorem}\label{thm:free-quotient}
Let $X$ be an $\mathcal{A}$-lattice. Suppose that
\renewcommand{\labelenumi}{(\alph{enumi})}
\begin{enumerate}
\item $X$ is a locally free $\mathcal{A}$-lattice of rank $d$, and
\item for each $i$, there exist $\beta_{i,1}, \ldots, \beta_{i,d}$ such that
$\mathcal{M}_{i}X = \mathcal{M}_{i}\beta_{i,1} \oplus \cdots \oplus 
 \mathcal{M}_{i}\beta_{i,d}$.
\end{enumerate}
Then $X$ is free of rank $d$ over $\mathcal{A}$ if and only if
\begin{enumerate}
\setcounter{enumi}{2}
\item there exist $\lambda_{i} \in U_{i}$ such that each $\alpha_{j} \in X$, where 
$\alpha_{j}  := \sum_{i=1}^{r} \alpha_{i,j}$ \\ and
$(\alpha_{i,1}, \ldots , \alpha_{i,d})^{\mathrm{T}} := 
\lambda_{i} (\beta_{i,1}, \ldots , \beta_{i,d})^{\mathrm{T}}$.
\end{enumerate}
Further, when this is the case, 
$X = \mathcal{A}\alpha_{1} \oplus \cdots \oplus \mathcal{A}\alpha_{d}$.
\end{theorem}

\section{The main algorithm}\label{sec:the-alg}

Let $E$ be a number field and let $G$ be a finite group.
Let $\mathcal{A}$ be any $\mathcal{O}_{E}$-order of full rank in the group algebra $E[G]$ and let $X$ be a left $\mathcal{A}$-lattice. In this section, we give the outline of an algorithm based on Theorem \ref{thm:free-quotient} (i.e.\ \cite[Corollary 2.4]{MR2422318}) that 
either computes elements $\alpha_{1}, \ldots, \alpha_{d} \in X$
such that  $X = \mathcal{A} \alpha_{1} \oplus \cdots \oplus \mathcal{A} \alpha_{d}$,
or determines that no such elements exist. In other words, the algorithm 
determines whether $X$ is free over $\mathcal{A}$, and if so, computes
explicit generators. 

We require the hypotheses (H1) and (H2\textprime) formulated in the introduction. 
We discuss the conditions under which these hypotheses hold in \S \ref{sec:hypotheses}. 
The sketch of the algorithm given here is essentially the same as \cite[Algorithm 3.1]{MR2422318}; 
the main work in the present article is in generalising the detailed versions of steps (5) and (7).

We assume that both $\mathcal{A}$ and $X$ are given by $\mathcal{O}_{E}$-pseudo-bases
as described, for example, in \cite[Definition 1.4.1]{MR1728313}. 
In other words, $X = \mathfrak{a}_{1} w_{1} \oplus \cdots \oplus \mathfrak{a}_{m} w_{m}$ 
where each $\mathfrak{a}_{i}$ is fractional ideal of $\mathcal{O}_{E}$ and each 
$w_{i} \in V := E \otimes_{\mathcal{O}_{E}} X$. 
Similarly, 
$\mathcal{A} = \mathfrak{b}_{1} \lambda_{1} \oplus \cdots \oplus \mathfrak{b}_{n} \lambda_{n}$
with fractional $\mathcal{O}_{E}$-ideals $\mathfrak{b}_{i}$ and $\lambda_{i} \in E[G]$.
Furthermore, we assume that $V$ is given by an $E$-basis $v_{1}, \ldots, v_{m}$ together with
matrices $A(\sigma) \in GL_{m}(E)$ for each $\sigma \in G$ describing the action of $G$ with respect
to $v_{1}, \ldots, v_{m}$.

\begin{algorithm}\label{alg:the-alg}
Input: $\mathcal{A}$ and $X$ as above.
\renewcommand{\labelenumi}{(\arabic{enumi})}
\begin{enumerate}
\item Compute $d := \dim_{E} (V)/|G|$ and check that $d \in \N$.

\item Compute a maximal $\mathcal{O}_{E}$-order $\mathcal{M}$ in $E[G]$ containing $\mathcal{A}$.

\item Compute the central primitive idempotents $e_{i}$ and the components 
$\mathcal{M}_{i}:=\mathcal{M} e_{i}$.

\item Compute the conductor $\mathfrak{c}$ of $\mathcal{A}$ in $\mathcal{M}$ and the components
$\mathfrak{c}_{i} := \mathfrak{c} e_{i}$. \\ Then compute the ideals 
$\mathfrak{g}_{i} := \mathfrak{c}_{i} \cap \mathcal{O}_{E_{i}}$ and $\mathfrak{f}_{i} := \frg_{i} \mathcal{M}_{i}$ 
for each $i$.

\item For each $i$, we try to compute $\beta_{i,1}, \ldots, \beta_{i,d}$ such that 
$\mathcal{M}_{i} X = 
\mathcal{M}_{i} \beta_{i,1} \oplus \cdots \oplus \mathcal{M}_{i} \beta_{i,d}$. If
such $\beta_{i,1}, \ldots, \beta_{i,d}$ do not exist, we terminate with 
`$\mathcal{M}X$ not free over $\mathcal{M}$'.

\item Check that $X$ is locally free of rank $d$ over $\mathcal{A}$.

\item For each $i$, compute a set of representatives 
$U_{i} \subset GL_{d}(\mathcal{M}_{i})^{}$ 
of the image of the natural projection map 
$GL_{d}(\mathcal{M}_{i})^{} \longrightarrow 
GL_{d}(\overline{\mathcal{M}_{i}})^{}$,
where $\overline{\mathcal{M}_{i}} := \mathcal{M}_{i} / \mathfrak{f}_{i}$.

\item Try to find  a tuple $(\lambda_{i}) \in \prod_{i=1}^{r} U_{i}$ such that 
that each $\alpha_{j} \in X$, where 
$\alpha_{j} := \sum_{i=1}^{r} \alpha_{i,j}$ \\ and
$(\alpha_{i,1}, \ldots , \alpha_{i,d})^{\mathrm{T}} := 
\lambda_{i} (\beta_{i,1}, \ldots , \beta_{i,d})^{\mathrm{T}}$.
For such a tuple, 
$X = \mathcal{A}\alpha_{1} \oplus \cdots \oplus \mathcal{A}\alpha_{d}$.
If no such tuple exists terminate with `$X$ not free over $\mathcal{A}$'.

\end{enumerate}
\end{algorithm}

Before commenting on the individual steps, we remark that steps (1) to (4) can be performed
in full generality without assuming hypotheses (H1) or (H2\textprime).

\renewcommand{\labelenumi}{(\arabic{enumi})}
\begin{enumerate}

\item If we replace $E[G]$ by a finite-dimensional semisimple $E$-algebra $A$ (see Remark \ref{rmk:fdss}), 
then we define $d := \dim_{E} (E \otimes_{\mathcal{O}_{E}} X)/ \dim_{E} (A)$. 

\item
An algorithm for computing $\mathcal{M}$ is described in
\cite[Kapitel 3 and 4]{friedrichs}.

\item
Each central primitive idempotent corresponds to an irreducible $E$-character $\chi_{i}$ 
and we have $e_{i} = \frac{n_{i}}{|G|} \sum_{g \in G} \chi_i(g^{-1})g$ with 
$n_{i}=\chi_{i}(1)$. If we replace $E[G]$ by a finite-dimensional semisimple $E$-algebra $A$, 
then we can use the algorithm of \cite[\S 2.4]{EberlyThesis}.

\item
In practise, we compute some multiple of the conductor. For example, one can use 
the method outlined in \cite[3.2 (f) and (g)]{MR2282916}. Also see \cite[Remark 3.3]{MR2282916}.

\item
This step is described in \S \ref{sec:mods-over-max}, using the results of \S \ref{sec:alg-max-div}.

\item 
Successful completion of step (5) shows that $\calM X$ is a free $\calM$-module
of rank $d$. Therefore $X$ is locally free of rank $d$ over $\mathcal{A}$
except possibly at the (finite number of) primes of $\mathcal{O}_{E}$ dividing the 
generalised module index $[\mathcal{M}:\mathcal{A}]_{\mathcal{O}_{E}}$
(if $\mathcal{O}_{E}[G] \subseteq \mathcal{A}$, then all such primes must divide $|G|$).
An algorithm to compute local basis elements (and thus to check local freeness) 
at these primes is given in \cite[\S 4.2]{MR2564571}. 

\item
This step is described in \S \ref{sec:enunits}.

\item
The number of tests for this step can be greatly reduced 
by using the method described in \cite[\S 7]{MR2422318}.
\end{enumerate}

\begin{remark}
Suppose $X$ is a finitely generated $\mathcal{O}_{E}[G]$-module in a free $E[G]$-space 
$V = E \otimes_{\mathcal{O}_{E}} X$ (for example, $L/K$ is a finite Galois extension of number fields
with $E \subseteq K$, $G=\Gal(L/K)$ and $X=\mathcal{O}_{L}$). Then it is often necessary to 
first compute the associated order 
$\mathcal{A} = \mathcal{A}(E[G]; X) := \{ \lambda \in E[G] \mid \lambda X \subseteq X \}$,
over which we wish to work (this is the only $\mathcal{O}_{E}$-order in $E[G]$ over which $X$ can possibly be free). 
This can be done by the method described in \cite[\S 4]{MR2422318}.
\end{remark}

\section{Hypotheses (H1) and (H2\textprime)}\label{sec:hypotheses}

We recall and discuss the hypotheses (H1) and (H2\textprime) required for Algorithm \ref{alg:the-alg}.

\subsection{Explicitly computing Wedderburn decompositions - (H1)} We note that satisfying
(H1) is equivalent to explicitly finding all irreducible $E[G]$-modules up to isomorphism. 
If $G$ is abelian, the required isomorphism can be explicitly computed using the character table. 
For $G$ non-abelian, many decompositions can be found in the literature or `by hand'. 
In general, explicitly computing Wedderburn decompositions of group algebras is a problem of major interest in its own right; we refer the reader to the Magma documentation for a survey of some of the methods currently implemented.

\subsection{The Eichler condition}\label{subsec:Eichler}
Let $F$ be a number field and let $A$ be a central simple $F$-algebra.
We say that $A$ satisfies the Eichler condition relative to $\mathcal{O}_{F}$ 
and write `$A$ is Eichler/$\mathcal{O}_{F}$' if and only if $A$ is \emph{not} a totally definite quaternion algebra (see \cite[(45.5)(i)]{MR892316} or \cite[(34.4)]{MR1972204}). 
More generally, if $A$ is a finite-dimensional semisimple $F$-algebra
we say that $A$ is Eichler/$\mathcal{O}_{F}$ if and only if each Wedderburn component 
$A_{i}$ is Eichler/$\mathcal{O}_{F_{i}}$, where $F_{i}$ is the centre of $A_{i}$.

\subsection{The locally free cancellation property - (H2\textprime)(a)}\label{subsec:loc-free-cancel}

Let $F$ be a number field and let $\Lambda$ be an $\mathcal{O}_{F}$-order in a finite-dimensional semisimple  $F$-algebra $A$. Then we say that $\Lambda$ has locally free cancellation if for any locally free finitely generated left $\Lambda$-modules $X$ and $Y$ we have
\[
X \oplus \Lambda^{(k)} \cong Y \oplus \Lambda^{(k)} \textrm{ for some } k \implies X \cong Y.
\]

The Jacobinski Cancellation Theorem says that if $A$ is Eichler/$\mathcal{O}_{F}$ then $\Lambda$ has locally free cancellation (see \cite[(51.24)]{MR892316}). 
It therefore remains to consider the case where $A$ is not Eichler/$\mathcal{O}_{F}$, i.e.\ at least one of the Wedderburn components $A_{i}$ is a totally definite quaternion algebra. 

We now restrict to the case that $\Lambda$ is a maximal $\mathcal{O}_{F}$-order in $A$. 
Note that $\Lambda$ has locally free cancellation if and only if all of 
the corresponding maximal orders $\Lambda_{i}$
in each Wedderburn component $A_{i}$ have locally free cancellation (for example, use Fr\"ohlich's result \cite[(51.26)]{MR892316}). Hence we may restrict further to the case that $A$ is a totally definite quaternion algebra and use the complete classification of maximal orders in such algebras with locally free cancellation given in \cite{MR2244802}. However, we take a different approach better suited to consideration of group algebras.

Let $E$ be a number field and $G$ be a finite group. 
We wish to give criteria for $E[G]$ to satisfy (H2\textprime)(a) in terms of conditions on $G$. 
Fix a Wedderburn component $M_{n}(D)=M_{n_{\chi}}(D_{\chi})$ of $E[G]$. 
Let $F$ be the centre of $D$ and fix a maximal $\mathcal{O}_{F}$-order $\Delta \subseteq D$.
Recall that (H2\textprime)(a) requires that $\Delta$ has locally free cancellation if $nd>1$. 
(Note that this is independent of the choice of $\Delta \subseteq D$ - see \cite[Proposition 9]{MR2244802}.)
If $n>1$, then $M_{n}(D)$ is Eichler/$\mathcal{O}_{F}$ and so any $\mathcal{O}_{F}$-order in
$M_{n}(D)$ has locally free cancellation;
however, we still require that $\Delta \subseteq D$ has locally free cancellation, which is not necessarily the case. 
On the other hand, assuming that $d>1$, a necessary condition for (H2\textprime)(a) to hold is that any maximal $\mathcal{O}_{E}$-order in $E[G]$ has locally free cancellation. For any $d \in \N$, this condition is sufficient if $n_{\chi}=1$
for each $\chi$ such that $D_{\chi}$ is a totally definite quaternion algebra.

Let $\Lambda$ be a maximal $\mathcal{O}_{E}$-order in $E[G]$. In light of the above discussion, 
we consider criteria on $G$ for $\Lambda$ to have locally free cancellation.
Let $Q_{4n}$ denote the generalised quaternion group of order $4n$,
and let $E_{24}, E_{48}, E_{120}$ denote the binary tetrahedral, octahedral and icosahedral 
groups of orders $24$, $48$ and $120$, respectively.
Then by \cite[(51.3)]{MR892316} (where $Q_{4n}$, $E_{24}, E_{48}, E_{120}$ are denoted by $Q_{n}, \tilde{T}, \tilde{O}, \tilde{I}$, respectively) 
$E[G]$ is Eichler/$\mathcal{O}_{E}$ 
(and so $\Lambda$ has locally free cancellation) if $G$ has no quotient isomorphic to 
$Q_{4n}$ ($n \geq 2$), $E_{24}$, $E_{48}$, or $E_{120}$.
In the case $E=\Q$, we have the following result due to Swan.

\begin{theorem}[{\cite[Theorem II]{MR703486}}]\label{thm:swan-max-cancel}
Let $G$ be a binary polyhedral group and let $\Lambda$ be a maximal order in $\Q[G]$.
Then $\Lambda$ has locally free cancellation if and only if $G$ is one of the
following $11$ groups:
$
Q_{8}, Q_{12}, Q_{16}, Q_{20}, Q_{24}, Q_{28}, Q_{36}, Q_{60}, E_{24}, E_{48}, E_{120}.
$
\end{theorem}

This leads to the following useful result.

\begin{lemma}\label{lemma:32-cancellation}
Let $G$ be any group with $|G|<32$. Then $\Q[G]$ satisfies \emph{(H2\textprime)(a)}.
\end{lemma}

\begin{proof}
Let $G$ be a group such that $\Q[G]$ has a Wedderburn component $M_{n_{\chi}}(D_{\chi})$ where $D_{\chi}$ is a totally definite quaternion algebra (for all other groups, the assertion follows from the Jacobinksi Cancellation Theorem) and $|G|<32$. 
It is straightforward to check using Magma that when $D_{\chi}$ is a totally definite quaternion algebra, we have $n_{\chi}=1$ (in the case that $G$ is a generalised quaternion group, this also follows from \cite[(7.40)]{MR632548}). 
Hence, by the discussion above, it suffices to show that any maximal order in $\Q[G]$ has locally free cancellation. 
So if $G$ is $Q_{8}, Q_{12}, Q_{16}, Q_{20}, Q_{24}, Q_{28}$, or $E_{24}$, the assertion now follows from Theorem \ref{thm:swan-max-cancel}.

The remaining possibilities for $G$ (determined using Magma) are $C_{2} \times Q_{8}$,
$C_{3} \times Q_{8}$, $C_{2} \times Q_{12}$, $S_{16,4}$, and $S_{24,1}$, 
where  $S_{n,i}$ denotes the group returned by the Magma function \texttt{Smallgroup(n,i)}.
In the last two cases, the quaternion component comes from surjections 
$S_{16,4} \twoheadrightarrow Q_{8}$ and $S_{24,1} \twoheadrightarrow Q_{12}$.
Hence the result now follows by combining Theorem \ref{thm:swan-max-cancel} and the fact that locally free cancellation for a maximal order $\Lambda$ in $\Q[G]$ is equivalent to locally free cancellation for each Wedderburn component $\Lambda_{i}$ (see discussion above).
\end{proof}

\begin{remark}\label{rmk:q8-cancellation}
By Theorem \ref{thm:swan-max-cancel}, maximal orders in $\Q[Q_{32}]$ do not have locally free cancellation. Hence $\Q[Q_{32}]$ does not satisfy (H2\textprime)(a) when $d>1$; 
however (H2\textprime)(a) is satisfied when $d=1$ since the only Wedderburn component 
$M_{n_{\chi}}(D_{\chi})$ with $D_{\chi}$ a totally definite quaternion algebra has $n_{\chi}=1$,
and so locally free cancellation is not required since $n_{\chi}d=1$.
\end{remark}

\subsection{Surjectivity of the reduced norm map - (H2\textprime)(b)}\label{subsec:norm}

Let $F$ be a number field and let $A$ be central simple $F$-algebra.
Let $\nr=\nr_{A/F}: A \longrightarrow F$ denote the reduced norm map 
as defined in \cite[\S 7D]{MR632548} or \cite[\S 9]{MR1972204}.
Let $\mathbb{H}$ be the skew field of real quaternions. 
Let $P$ be a real prime of $F$, let $A_{P}$ be the completion of $A$ at $P$, and let $\sigma_{P}: F \hookrightarrow \R$ be the corresponding embedding. We say that $P$ is ramified in $A$ if and only if $A_{P}$ is isomorphic to $M_{n}(\mathbb{H})$ for some $n \in \N$. We define
\[
U(A) := \{  \alpha \in F \mid \sigma_{P}(\alpha) > 0 
\textrm{ for every real prime $P$ of $F$ ramified in $A$} \}.
\]

The Hasse-Schilling-Maass Norm Theorem 
(see \cite[(33.15)]{MR1972204} or \cite[(7.48)]{MR632548}) says that $\nr(A^{\times}) = U(A)$. 
Now let $\Lambda$ be a maximal $\mathcal{O}_{F}$-order in $A$. Then
$\nr(\Lambda^{\times}) \subseteq \mathcal{O}_{F}^{\times +} := \mathcal{O}_{F}^{\times} \cap U(A)$.
Note that
$(\mathcal{O}_{F}^{\times})^{2} \subseteq \mathcal{O}_{F}^{\times +}$ and so 
$\mathcal{O}_{F}^{\times +}$ is a subgroup of index some power of $2$ in $\mathcal{O}_{F}^{\times}$ which can easily be computed (provided $\mathcal{O}_{F}^{\times}$ can be computed).

The question of whether $\nr(\Lambda^{\times}) = \mathcal{O}_{F}^{\times +}$ is directly relevant to hypothesis (H2\textprime)(b). If $A$ is Eichler/$\mathcal{O}_{F}$ then by \cite[(51.22)]{MR892316} we in fact have 
equality. However, if $A$ is not Eichler/$\mathcal{O}_{F}$ then we may or may not have equality. For example, if $F=\Q$ and $A$ is a totally definite quaternion algebra, then we have 
$\mathcal{O}_{F}^{\times +}=\{1\}$ and so equality is clear.  On the other hand, \cite[p.198-199]{MR584612} gives an example in which $A$ is a totally definite quaternion algebra over its centre 
$F=\Q(\sqrt{3})$ for which equality does not hold. 

\begin{lemma}\label{lemma:40-nr}
Let $G$ be any group with $|G|<40$. 
Let $M_{n}(D)=M_{n_{\chi}}(D_{\chi})$ be a Wedderburn component of $\Q[G]$, let $F$ be the centre of $D$, and let $\Delta \subseteq D$ be any maximal $\mathcal{O}_{F}$-order. 
Then $\nr(\Delta^{\times}) = \mathcal{O}_{F}^{\times +}$.
In particular, $\Q[G]$ satisfies \emph{(H2\textprime)(b)}.
\end{lemma}

\begin{proof}
See Magma sample file. 
\end{proof}

\begin{remark}\label{rmk:H2b-depends-on-choice}
When $G=Q_{40}$, the generalised quaternion group of order $40$, there are two  
$D_{\chi}$ which are totally definite quaternion: one with centre $\Q$; the other with centre 
$F:=\Q(\zeta_{20})^+$, the maximal totally real subfield of $\Q(\zeta_{20})$.
In the latter case, there are three maximal orders $\Delta \subseteq D_{\chi}$, 
only two of which satisfy $\nr(\Delta^{\times}) = \mathcal{O}_{F}^{\times +}$. 
\end{remark}

\subsection{Computing unit groups of maximal orders in skew fields - (H2\textprime)(c)}\label{subsec:norm-unit-reps}

Let $D$ be a skew field that is central and finite-dimensional over a number field $F$.
Let $\Delta \subseteq D$ be a maximal $\mathcal{O}_{F}$-order. 
The unit group $\Delta^{\times}$ is always finitely presentable (see \cite{MR1309127}, for example). We consider the problem of computing a set of generators of $\Delta^{\times}$.

If $D$ is commutative (i.e.\ $D=F$) then a set of generators of 
$\Delta^{\times}=\mathcal{O}_{F}^{\times}$ is computable by \cite[Algorithm 6.5.8]{MR1228206}
(the Magma command is \texttt{UnitGroup}).
When $D$ is a totally definite quaternion algebra then
in fact $[\Delta^{\times} : \mathcal{O}_{F}^{\times}] < \infty$. 
In this case, the Magma command $\texttt{Units}$ computes a set of representatives of 
$\Delta^{\times}/\mathcal{O}_{F}^{\times}$ and thus reduces the problem to the previous case
(also see \cite[Remark 7.5]{MR2592031}). 

The authors are unaware of any algorithms to compute a set of generators of $\Delta^{\times}$ in other cases. However, when $\nr(\Delta^{\times})=\mathcal{O}_{F}^{\times +}$ (which by the discussion in \S \ref{subsec:norm} must be the case whenever $D$ is not a totally definite quaternion algebra), it suffices for our purposes to solve the following somewhat easier problem: 
compute a set of representatives of the map $\nr: \Delta^{\times} \longrightarrow  \mathcal{O}_{F}^{\times +}$, i.e.\ compute a finite subset $S \subseteq \Delta^{\times}$ such that $\nr(S)$ generates $\mathcal{O}_{F}^{\times +}$. Note that, in particular, $S$ can be taken to be a set of generators of $\Delta^{\times}$.

We note that (H2\textprime)(c) is satisfied whenever $D=D_{\chi}$ is commutative or a totally definite quaternion algebra: in this case, we can always compute a set of generators of $\Delta^{\times}$ and hence we can compute a set of representatives of $\nr: \Delta^{\times} \longrightarrow  \mathcal{O}_{F}^{\times +}$ when $\nr(\Delta^{\times})=\mathcal{O}_{F}^{\times +}$.

\subsection{The principal ideal problem for maximal orders in skew fields - (H2\textprime)(d)}\label{subsec:PIB}

Let $D$ be a skew field that is central and finite-dimensional over a number field $F$.
Let $\Delta \subseteq D$ be a maximal $\mathcal{O}_{F}$-order and let $\mathfrak{a}, \mathfrak{b}$ 
be fractional left $\Delta$-ideals. Then we say that we can solve the
principal ideal problem for left ideals if we have an algorithm to 
\renewcommand{\labelenumi}{(\roman{enumi})}
\begin{enumerate}
\item decide whether $\mathfrak{a} \cong \mathfrak{b}$ as left $\Delta$-ideals; and
\item if $\mathfrak{a} \cong \mathfrak{b}$, compute $\xi \in D$ such that $\mathfrak{a} = \mathfrak{b} \xi$.
\end{enumerate}
Dually, we may formulate the principal ideal problem for right ideals.

If $D$ is commutative (i.e.\ $D=F$) then the problem is solved by \cite[Algorithm 6.5.10]{MR1228206}
and implemented in Magma.
The algorithms in \cite{MR2592031} solve the principal ideal problem when $D$ is any  quaternion algebra, and are implemented in Magma when $F$ is totally real.
(In both cases, the relevant Magma command is \texttt{IsPrincipal}.)

The authors are unaware of any algorithms solving this problem completely in other cases. 
However, if $D$ is Eichler/$\mathcal{O}_{F}$, then by \cite[(34.9)]{MR1972204} a left $\Delta$-ideal $\mathfrak{a}$ is principal if and only if $\nr(\mathfrak{a})$ is a principal $\mathcal{O}_{F}$-ideal with a generator $\alpha \in U(D)$, solving (i). 
Furthermore, it is straightforward to show that for any $D$ (not necessarily Eichler/$\mathcal{O}_{F}$) a left $\Delta$-ideal $\mathfrak{a}$ is principal if and only if there exists $\xi \in \mathfrak{a}$ such that $\nr(\mathfrak{a})=\nr(\xi)\mathcal{O}_{F}$.

\subsection{Choice of $\Delta \subseteq D$}
We note that (H2\textprime)(a) is independent of the choice of $\Delta \subseteq D$ 
(see \S \ref{subsec:loc-free-cancel}), but (H2\textprime)(b) is not 
(see Remark \ref{rmk:H2b-depends-on-choice}). Moreover, (H2\textprime)(c) and (H2\textprime)(d) are independent of the choice of $\Delta$ in the cases that they are known to hold (i.e.\ when $D$ is a number field or totally definite quaternion algebra). 
This is important because if $n=1$ then $\Delta$ is determined by $\mathcal{M}$,
 and the choice of $\mathcal{M}$ may be limited by the requirement that
$\mathcal{A} \subseteq \mathcal{M}$. 
On the other hand, if $n \geq 2$ then we can make any choice of $\Delta$.
(The differences between the $n=1$ and $n \geq 2$ cases are made clear in
\S \ref{subsec:step5-outline}.)

\subsection{Particular cases in which (H2\textprime) holds}

Let $E$ be a number field, let $G$ be a finite group, and let $d \in \N$.
We consider particular cases in which the pair ($E[G]$,$d$) satisfies (H2\textprime). 
We note that (H2\textprime) holds whenever (H2) holds, and the latter does not depend on $d$.

\begin{prop}\label{prop:h2prime-holds}
The pair ($E[G]$,$d$) satisfies \emph{(H2\textprime)} in the following cases:
\renewcommand{\labelenumi}{(\roman{enumi})}
\begin{enumerate}
\item $G$ is abelian, dihedral, or symmetric;
\item $G$ is a nilpotent group of odd order (e.g. $G$ is a $p$-group where $p$ is an odd prime);
\item $E$ contains a primitive $m$th root of unity, where $m$ is the exponent of $G$;
\item $G$ is a generalised quaternion group, $E=\Q$, and $d=1$;
\item $|G|<32$ and $E=\Q$.
\end{enumerate}
\end{prop}

\begin{proof}
In cases (i), (ii) and (iii), then in fact the stronger hypothesis (H2) holds (for a general discussion 
of Schur indices, see \cite[\S 74B]{MR892316}.) 
For (iv) the claim then follows from \cite[(7.40)]{MR632548}. 
Indeed, for any Wedderburn component $M_{n_{\chi}}(D_{\chi})$ of $\Q[G]$,
either $D_{\chi}$ is a number field or $n_{\chi}=1$ (so $n_{\chi}d=1$) 
and $D_{\chi}$ is a totally definite quaternion algebra. 
In case (v), it is straightforward to check using Magma that each $D_{\chi}$ is either a number field or
a totally definite quaternion algebra, so (H2\textprime)(c) and (d) are satisfied. 
Hypotheses (H2\textprime)(a) and (b) now follow from Lemmas \ref{lemma:32-cancellation} and \ref{lemma:40-nr}, respectively.
\end{proof}

\begin{remark}
Using the Magma commands \texttt{CharacterTable}, \texttt{SchurIndex} and \texttt{SchurIndices}, 
one can often check whether a particular pair $(\Q[G],d)$ satisfies (H2\textprime). 
For example, 
of the $1268$ groups $G$ with $|G|<128$, there are $433$ such that $(\Q[G],1)$ 
does not satisfy (H2), whereas only $181$ are such that $(\Q[G],1)$ does not satisfy (H2\textprime).
\end{remark}

\begin{remark}\label{rmk:slightly-weaker}
Algorithm \ref{alg:the-alg} can still be run in certain situations where the full strength of (H2\textprime) does not hold. For example, (b) and (c) are only needed for step (7) of Algorithm \ref{alg:the-alg} and so are unnecessary if $\mathcal{A}=\mathcal{M}$. In the case that $\mathcal{A} \neq \mathcal{M}$,
if (b) is dropped and (c) is weakened to being able to compute a `large' subset of representatives of $\nr:\Delta^{\times} \longrightarrow \nr(\Delta^{\times}) \subseteq \mathcal{O}_{F}^{\times +}$, we may have enough respresentatives to successfully find generators (if they exist) in step (7); however, 
failure to find generators will not prove that $X$ is not free over $\mathcal{A}$.
Similarly, if (a) is dropped then it may still be possible to find generators over the maximal order -
see Remark \ref{rmk:drop-cancellation}.
Finally, we note that  (H2\textprime)(d) is always needed.
\end{remark}

\begin{remark}\label{rmk:fdss}
Algorithm \ref{alg:the-alg} still works if the group algebra $A := E[G]$ is replaced by a
finite-dimensional semisimple $E$-algebra, in which case analogous versions of (H1) and 
(H2\textprime) are required (also note the minor changes needed in steps (1) and (3) 
- see \S \ref{sec:the-alg}).
However, note that it will not be possible to apply the Grunwald-Wang Theorem if the `special case' occurs - see proof Lemma \ref{lemma:W-field} and Remark \ref{rmk:algorithmic-Grunwald-Wang}. 
Of course, the modified version of (H1) is not necessary if $A$ is given directly as a product of matrix rings.
\end{remark}

\section{Algorithms for modules over maximal orders in skew fields}\label{sec:alg-max-div}

\subsection{Algorithmic version of Roiter's Lemma}\label{subsec:alg-roiter}

Let $R$ be a Dedekind domain with field of fractions $F$. Let $\Lambda \subseteq A$ be any $R$-order in the finite-dimensional semisimple $F$-algebra $A$.
Recall that two $\Lambda$-lattices $M,N$ are in the same \emph{genus} (denoted $M \vee N$)
if for each prime ideal $\mathfrak{p}$ of $R$, there is a $\Lambda_{\mathfrak{p}}$-isomorphism
$M_{\mathfrak{p}} \cong N_{\mathfrak{p}}$.

\begin{theorem}[Roiter's Lemma]
Let $M,N$ be $\Lambda$-lattices. 
Then $M \vee N$ if and only if for any nonzero integral ideal $\mathfrak{a}$ of $R$ there exists an injective 
homomorphism of $\Lambda$-lattices $\varphi : M \hookrightarrow N$ such that 
$\mathfrak{a} + \ann_R(\coker \varphi) = R$.
\end{theorem}

\begin{proof}
See \cite[(27.1)]{MR1972204} or \cite[(31.6)]{MR632548}, for example.
\end{proof}

Let $M,N$ be locally free $\Lambda$-lattices (i.e.\ $M, N$ are both in the genus of $\Lambda$) and
let $\mathfrak{a}$ be a nonzero integral ideal of $R$. We wish to make Roiter's Lemma algorithmic in this situation, i.e.\ explicitly compute $\varphi$ such that $\mathfrak{a} + \ann_R(\coker \varphi) = R$.

Let $\mathfrak{p}_{1}, \ldots, \mathfrak{p}_{n}$ denote the prime divisors of $\mathfrak{a}$ 
(or choose any prime $\frp_1$ if $\mathfrak{a} = R$). 
By the method described in \cite[\S 4.2]{MR2564571}, for each $i=1, \ldots, n$ 
we compute local bases
\begin{eqnarray*}
  M_{\mathfrak{p}_{i}} &=& \Lambda_{\mathfrak{p}_{i}} \omega_{i,1} \oplus \cdots \oplus \Lambda_{\mathfrak{p}_{i}} \omega_{i,d}, \\
  N_{\mathfrak{p}_{i}} &=& \Lambda_{\mathfrak{p}_{i}} \nu_{i,1} \oplus \cdots \oplus \Lambda_{\mathfrak{p}_{i}} \nu_{i,d},
\end{eqnarray*}
set $\psi_{i}(\omega_{i,k})= \nu_{i,k}$ for $k = 1, \ldots, d$, and extend linearly.
Hence $\psi_{i} : M_{\frp_i} \longrightarrow N_{\frp_i}$ is an isomorphism of 
$\Lambda_{\mathfrak{p}_{i}}$-modules for $i = 1, \ldots, n$.  
By multiplying the basis elements $\nu_{i,k}$ by elements of $R_{\mathfrak{p}_{i}}^{\times}$ if necessary, we may assume that $\psi_i(M) \subseteq N$ for each $i$.

Following \cite[Exercise 18.3]{MR1972204}, we now compute $\beta_{1}, \ldots, \beta_{n} \in R$ 
such that
\begin{eqnarray*}
  \beta_i &\equiv& 1 \pmod {\frp_{i}}, \\
  \beta_j &\equiv& 0 \pmod {\frp_{j}} \textrm{ for } j \ne i,
\end{eqnarray*}
(one can use the $\texttt{CRT}$ function of Magma; also see \cite[Proposition 1.3.11]{MR1728313})
and set $\varphi:M \longrightarrow N$ to be the restriction of $\sum_{j=1}^n \beta_j \psi_j$.
By Nakayama's Lemma each localised map 
$\varphi_{\frp_i}: M_{\frp_i} \longrightarrow N_{\frp_i}$ is surjective. 
Since $M_{\mathfrak{p}_{i}}$ is $R_{\mathfrak{p}_{i}}$-torsion-free, a rank argument shows that each 
$\varphi_{\frp_i}$ is in fact an isomorphism.

We now follow the proof of \cite[(27.1)]{MR1972204}. Since
$( \ker \varphi )_{\mathfrak{p}_{i}} = \ker \varphi_{\mathfrak{p}_{i}} = 0$
for each $i$ and $\ker \varphi$ is $R$-torsion-free, we see that $\ker \varphi=0$ and so $\varphi$ is injective. 
Furthermore, $(\coker\varphi)_{\mathfrak{p}_{i}} = \coker \varphi_{\mathfrak{p}_{i}} = 0$ 
and so $\mathfrak{p}_{i} + \ann_{R}(\coker \varphi)=R$ for each $i$. 
Therefore $\mathfrak{a} + \ann_{R}(\coker \varphi)=R$.

\begin{remark}
We can replace the hypothesis that $M,N$ are locally free with $M \vee N$ by the assumption that 
we can explicitly compute isomorphisms 
$\psi_{i} : M_{\mathfrak{p}_{i}} \longrightarrow N_{\mathfrak{p}_{i}}$
for $i = 1, \ldots, n$. 
\end{remark}

The main application we have in mind is as follows. 
Let $D$ be a skew field that is central and finite-dimensional over a number field $F$.
Let $\Delta \subseteq D$ be a maximal $\mathcal{O}_{F}$-order. 
We take $\Lambda = \Delta$, $A = D$, $R=\mathcal{O}_{F}$, and $M,N$ to be fractional left $\Delta$-ideals. Then for each $i$ we have 
\[
M_{\mathfrak{p}_{i}} = \Delta_{\mathfrak{p}_{i}} \omega_{i}, \quad
N_{\mathfrak{p}_{i}} = \Delta_{\mathfrak{p}_{i}} \nu_{i}, \quad
\psi_{i}(\omega_{i}) = \nu_{i} = \omega_{i} (\omega_{i}^{-1}\nu_{i}),
\]
i.e.\ $\psi_{i}$ is right multiplication by $\xi_{i} := \omega_{i}^{-1}\nu_{i}$. 
The map $\varphi$ in the algorithmic version of Roiter's Lemma is then right multiplication by 
$\xi := \sum_{j=1}^{n} \beta_{j} \xi_{j}$.

\subsection{The noncommutative extended Euclidean algorithm}\label{subsec:ex-euclid}

Let $D$ be a skew field that is central and finite-dimensional over a number field $F$.
Let $\Delta \subseteq D$ be a maximal $\mathcal{O}_{F}$-order.
We briefly recall the following definitions and facts from \cite[\S 8 and \S 22]{MR1972204}.
Let $M$ be any full left $\mathcal{O}_{F}$-lattice in $D$ 
(for example, a fractional left ideal of $\Delta$). 
The \emph{right order} of M is defined to be
\[
\mathcal{O}_{r}(M) := \{ x \in D \mid M x \subseteq M \}.
\]
Then $\mathcal{O}_{r}(M)$ is an $\mathcal{O}_{F}$-order in $D$. The \emph{left order}
$\mathcal{O}_{l}(M)$ is defined analogously. We define
\[
M^{-1} := \{ x \in D \mid M \cdot x \cdot M \subseteq M \} 
\]
and note that this is also a full right $\Delta$-lattice in $D$. 
If $\Delta = \calO_l(M)$ and $\Delta' = \calO_r(M)$, then $\Delta' = \calO_l(M^{-1})$ and $\Delta = \calO_r(M^{-1})$.
By \cite[(22.7)]{MR1972204} we have
\begin{equation}\label{eq:inverse-ideals}
M \cdot M^{-1} = \Delta, \qquad M^{-1} \cdot M = \Delta', \qquad (M^{-1})^{-1} = M. 
\end{equation}

We consider the following problem. 
Let $\mathfrak{a}, \mathfrak{b}, \mathfrak{c}$ be fractional left $\Delta$-ideals such that 
$\mathfrak{a} + \mathfrak{b} = \mathfrak{c}$. 
We wish to find $\alpha \in \mathfrak{c}^{-1} \mathfrak{a}$ and $\beta \in \mathfrak{c}^{-1} \mathfrak{b}$
such that $\alpha+\beta =1$. Observe that $\mathfrak{c}^{-1} \mathfrak{a}$ and $\mathfrak{c}^{-1} \mathfrak{b}$ are left ideals over $\Delta':= \mathcal{O}_{r}(\mathfrak{c})$ and we
have
\[
\mathfrak{a} + \mathfrak{b} = \mathfrak{c} \iff \mathfrak{c}^{-1} \mathfrak{a} + \mathfrak{c}^{-1} \mathfrak{b} = \mathfrak{c}^{-1} \mathfrak{c} = \Delta'.
\]
 
We shall essentially give the argument of \cite[Algorithm 1.3.2]{MR1728313}, to which we refer the reader for more details. For background material on the Hermite Normal Form (henceforth abbreviated to HNF) over $\Z$, see \cite[\S 2.4]{MR1228206}. 
Let $\omega_{1}', \ldots, \omega_{n}'$ be a $\Z$-basis of $\Delta'$ chosen so that $\omega_{1}'=1$.
Then the underlying $\Z$-modules of $\mathfrak{c}^{-1}\mathfrak{a}$, 
$\mathfrak{c}^{-1}\mathfrak{b}$ are given by
HNFs over $\Z$ with respect to $\{ \omega_{1}', \ldots, \omega_{n}' \}$, say $H_{\mathfrak{a}}, H_{\mathfrak{b}} \in M_{n \times n}(\Z)$.
Consider the matrix $(H_{\mathfrak{a}} \mid H_{\mathfrak{b}})$. Then by 
\cite[\S 2.4.2]{MR1228206} we can compute $U \in GL_{2n}(\Z)$ such that
\[
(H_{\mathfrak{a}} \mid H_{\mathfrak{b}} )U = (0 \mid H),
\]
where $H$ is the HNF of $(H_{\mathfrak{a}} \mid H_{\mathfrak{b}})$. Then $H$ must be the identity matrix since $\mathfrak{c}^{-1} \mathfrak{a} + \mathfrak{c}^{-1} \mathfrak{b} = \Delta'$. 
Let $Z = U_{n+1}$ be the
$(n+1)$-st column of $U$. Then
\[
(H_{\mathfrak{a} } \mid H_{\mathfrak{b}}) Z = (H_{\mathfrak{a}} \mid H_\mathfrak{b})
\left(
  \begin{array}{c}
    Z_{\mathfrak{a}}\\Z_{\mathfrak{b}} 
  \end{array}\right) = H_{\mathfrak{a}} Z_{\mathfrak{a}} + H_{\mathfrak{b}} Z_{\mathfrak{b}} 
  =: z_{\mathfrak{a}} + z_{\mathfrak{b}}.
\]
The column vectors $z_{\mathfrak{a}}$ and $z_{\mathfrak{b}}$ correspond to 
$\alpha \in \mathfrak{c}^{-1}\mathfrak{a}$ and $\beta \in \mathfrak{c}^{-1} \mathfrak{b}$.

We now consider the following slightly more general problem. 
Let $\mathfrak{a}_{1}, \ldots, \mathfrak{a}_{n}, \mathfrak{c}$ be fractional left $\Delta$-ideals such that 
$\mathfrak{a}_{1} + \cdots + \mathfrak{a}_{n}  = \mathfrak{c}$. We wish to compute $\alpha_j \in \frc^{-1}\mathfrak{a}_j$ such that
\[
\alpha_{1} + \cdots + \alpha_{m} = 1.
\]
Let $\mathfrak{b} = \mathfrak{a}_{1} + \cdots + \mathfrak{a}_{m-1}$. 
Assume that we have $\beta_{j} \in \mathfrak{b}^{-1} \mathfrak{a}_{j}$ such that
\[
\beta_{1} + \cdots + \beta_{m-1} = 1.
\]
By the above we can find $\xi \in \frc^{-1}\mathfrak{b}$ and $\eta \in \frc^{-1} \mathfrak{a}_m$ with $\xi + \eta = 1$. Then
\[
1 = \xi ( \beta_{1} + \cdots + \beta_{m-1} ) + \eta = \xi\beta_{1} + \cdots + \xi \beta_{m-1} + \eta
\]
and $\xi\beta_{j} \in \frc^{-1}\mathfrak{b}\mathfrak{b}^{-1}\mathfrak{a}_{j} = \frc^{-1}\mathfrak{a}_j$. 
Hence we are reduced to the case $m=2$ solved above.

\begin{remark}
It is straightforward to give an analogous `right version' of this algorithm.
\end{remark}

\subsection{Noncommutative Hermite Normal Forms}\label{subsec:ncHNF}

Let $D$ be a skew field that is central and finite-dimensional over a number field $F$.
Let $\Delta \subseteq D$ be a maximal $\mathcal{O}_{F}$-order.
Let $X$ be a left $\Delta$-lattice such that $FX \cong D^{r}$, for some $r > 0$. 
By \cite[(2.44) and (2.45)(ii)]{MR1972204} there exist 
$x_1, \ldots , x_r \in FX$ and fractional left $\Delta$-ideals $\mathfrak{a}_{1}, \ldots, \mathfrak{a}_{r}$
such that
\[
X = \mathfrak{a}_{1} x_{1} \oplus \cdots \oplus \mathfrak{a}_{r} x_{r}.
\]
Our aim is to explicitly compute such a noncommutative pseudo-basis under the 
assumption that we have the following:
\renewcommand{\labelenumi}{(\roman{enumi})}
\begin{enumerate}
\item a $\Z$-basis $\omega_{1}, \ldots, \omega_{n}$ for $\Delta$ with $\omega_{1}=1$;
\item a left $D$-basis $v_{1}, \ldots, v_{r}$ for $FX$, i.e.\ $FX = Dv_{1} \oplus \cdots \oplus D v_{r}$;
and
\item $y_{1}, \ldots , y_{k} \in FX$ and fractional left $\Delta$-ideals $\mathfrak{b}_{i}$ such that 
$X = \mathfrak{	b}_{1} y_{1} + \cdots + \mathfrak{b}_{k} y_{k}$.
\end{enumerate}
Note that we must have $k \geq r$. We write 
\[
y_{j} = \sum_{i=1}^r a_{ij}v_i \text{ with } a_{ij} \in D
\]
and set
\[
A := \left( a_{ij} \right) \in M_{r \times k}(D).
\]
Then we have a `pseudo-matrix' $\mathcal{A} := (A, (\mathfrak{b}_{1}, \ldots, \mathfrak{b}_{k}))$ representing $X$.
We give noncommutative versions of some of the results of \cite[\S 1.4.2]{MR1728313} to transform 
 $\mathcal{A}$ to a HNF over $\Delta$ (the key difference being that one needs to consider carefully whether the required multiplications are on the left or the right).
In other words, we compute a pseudo-matrix
 \[
 \left( (0 \mid H), (\mathfrak{a}_{1}, \ldots, \mathfrak{a}_{r}) \right) \textrm{ where }
 H = \left(
 \begin{array}{cccc}
 1 & * & \cdots & * \\
 0 & 1 & \cdots & * \\
 \vdots & \ddots & \ddots & \vdots \\
 0 & \cdots & 0 & 1
 \end{array}
 \right)
 \in M_{r \times r}(D)
 \]
 such that
\[
\mathfrak{b}_{1} A_{1} + \cdots + \mathfrak{b}_{k} A_{k} = \mathfrak{a}_{1} H_{1} \oplus \cdots \oplus \mathfrak{a}_{r} H_{r},
\]
where $A_j$ and $H_j$ denote the $j$th columns of the matrices $A$ and $H$ respectively.
(Note that the sum $\mathfrak{a}_{1} H_{1} + \cdots + \mathfrak{a}_{r} H_{r}$ must be direct 
since the $H_{j}$'s are clearly linearly independent.)

We describe the first step, the rest being induction. We set
\[
\mathfrak{b}_{j}' :=
\begin{cases}
  \mathfrak{b}_{j} a_{rj}, & \text{ if } a_{rj} \ne 0, \\
  \mathfrak{b}_{j},        & \text{ if } a_{rj} = 0,  
\end{cases}
\quad \textrm{ and } \quad 
a_{ij}' :=
\begin{cases}
  a_{rj}^{-1}a_{ij}, & \text{ if } a_{rj} \ne 0, \\
  a_{ij}, & \text{ if } a_{rj} = 0.
\end{cases}
\]
By relabelling (removing the $'$ notation) and reordering if necessary, we may therefore assume that 
$\mathcal{A}$ is of the form
\[
\left( \left( B_{1} \mid B_{2} \right) , (\mathfrak{b}_{1}, \ldots, \mathfrak{b}_{k}) \right)
\textrm{ where } B_{1} = \left(
\begin{array}{ccc}
* & \ldots & * \\
\vdots & \ddots & \vdots  \\
* & \cdots  & *  \\
0 & \cdots & 0
\end{array} \right)
\textrm{ and } 
B_{2} = 
\left(
\begin{array}{ccc}
* & \cdots & * \\
\vdots &  \ddots & \vdots \\
* & \cdots & * \\
1 & \cdots & 1
\end{array} \right).
\]
It suffices to consider the matrix $B_{2}$, so without loss of generality we may assume that $A=B_{2}$
and $\mathcal{A} = ( A, (\mathfrak{b}_{1}, \ldots, \mathfrak{b}_{k}))$.

We explicitly compute  $\mathfrak{c} := \mathfrak{b}_{1} + \cdots + \mathfrak{b}_{k}$ by HNF techniques over $\Z$ (see \cite[\S 2.4]{MR1228206}). Using \S \ref{subsec:ex-euclid}, we then compute
$\alpha_{j} \in \mathfrak{c}^{-1} \mathfrak{b}_{j}$ such that
\[
\alpha_{1} + \cdots + \alpha_{k} = 1.
\]
Let $c := \alpha_{1} A_{1} + \cdots + \alpha_{k} A_{k}$ and let $A':= ( A_{1} - c, \ldots, A_{k} - c , c ) \in M_{r \times (k+1)}(D)$, the matrix formed by column vectors in the obvious way.
We consider the pseudo-matrix 
\[
\mathcal{A}' := \left( A', (\mathfrak{b}_1, \ldots, \mathfrak{b}_k, \frc) \right).
\]
For a pseudo-matrix $\mathcal{C} = (C , (\mathfrak{c}_{1}, \ldots, \mathfrak{c}_{m}))$ with $C \in M_{n \times m}(D)$ we set 
\[
\langle \mathcal{C} \rangle := \mathfrak{c}_1 C_{1} + \cdots + \mathfrak{c}_{m} C_{m}.
\]

\begin{lemma}
We have  $\langle \mathcal{A} \rangle = \langle \mathcal{A}' \rangle$.
\end{lemma}

\begin{proof}
`$\subseteq$' Let $s \in \mathfrak{b}_j$. Then $sA_j = s(A_j - c) + sc \in \langle \mathcal{A'} \rangle$
because $s \in \mathfrak{b}_{j} \subseteq \frc$.

`$\supseteq$' Again let $s \in \mathfrak{b}_j$. Then 
\begin{eqnarray*}
  s(A_j -c) \in \langle \mathcal{A} \rangle &\iff& sc \in \langle \mathcal{A} \rangle
\iff s(\alpha_{1}A_{1} + \cdots + \alpha_{k} A_{k}) \in \langle \mathcal{A} \rangle \\
&\Longleftarrow& s\alpha_{j} \in \mathfrak{b}_j \text{ for } j = 1, \ldots, k.
\end{eqnarray*}
Since $\alpha_j \in \frc^{-1}\mathfrak{b}_j$ and $s \in \mathfrak{b}_j \subseteq \frc$, we have $s\alpha_j \in \frc\frc^{-1}\mathfrak{b}_j = \mathfrak{b}_j$.

Now suppose $s \in \frc$. Since $\alpha_j \in \mathfrak{c}^{-1}\mathfrak{b}_j$ each 
$s\alpha_j \in \mathfrak{b}_j$ for $j = 1, \ldots, k$. Hence $sc \in \langle \mathcal{A} \rangle$. 
\end{proof}

Finally, note that $A'$ is of the form
\[
 \left(
  \begin{array}{ccc}
    A_1 & \vline & * \\ \hline  0 & \vline & 1
  \end{array} \right)
\]
for some $A_{1} \in M_{(r-1)  \times k}(D)$. Hence we can now repeat the process with 
$(A_{1}, (\mathfrak{b}_{1}, \ldots, \mathfrak{b}_{k}))$ and continue inductively until we obtain a pseudo-matrix of the desired form.

\subsection{Noncommutative Steinitz form}\label{subsec:steinitz}
We assume the notation and setting of \S \ref{subsec:ncHNF}. The aim of this section is to give
an algorithmic version of \cite[(27.4)]{MR1972204}.
Given fractional left $\Delta$-ideals $\mathfrak{a}_{1}, \ldots, \mathfrak{a}_{r}$ and $x_{1}, \ldots, x_{r} \in FX$ such that
\[
X = \mathfrak{a}_{1} x_{1} \oplus \cdots \oplus \mathfrak{a}_{r} x_{r},
\]
we wish to compute a Steinitz form, i.e.\ a fractional left $\Delta$-ideal $\mathfrak{b}$
and $z_{1}, \ldots, z_{r} \in FX$ such that
\[
X = \Delta z_{1} \oplus \cdots \oplus \Delta z_{r-1} \oplus \mathfrak{b} z_{r}.
\]
(Note that without loss of generality we can in fact take $\mathfrak{b}$ to be integral.)
In general, we argue as follows
\begin{eqnarray*}
  X &=& \mathfrak{a}_{1} x_{1} \oplus \cdots \oplus \mathfrak{a}_{r} x_{r} \\
    &=& \Delta x_{1}' \oplus \mathfrak{b}_{2} x_{2}' \oplus \mathfrak{a}_{3} x_{3} \oplus \cdots \oplus 
    \mathfrak{a}_{r} x_{r} \\
    &=& \Delta x_{1}' \oplus \Delta x_{2}'' \oplus \mathfrak{b}_{3} x_{3}'' \oplus \mathfrak{a}_{4} x_{4} \oplus \cdots \oplus \mathfrak{a}_{r} x_{r} \\
    &=& etc.,
\end{eqnarray*}
so we may restrict to the case $r=2$.

For later reference we note the following lemma.

\begin{lemma}
Let $\mathfrak{a}, \mathfrak{b}$ be fractional left $\Delta$-ideals. Then
\[
\mathfrak{a} \cong \mathfrak{b} \text{ as left } \Delta-\text{modules} \iff \mathfrak{a} = \mathfrak{b} \xi \text{ for some } \xi \in D.
\]
\end{lemma}

\begin{proof}
Obvious, but pay attention to the fact that $\xi$ is on the right side of $\mathfrak{b}$.
\end{proof}

We first consider the special case that 
$X = \mathfrak{a}_{1} x_{1} \oplus \mathfrak{a}_{2} x_{2}$ with 
$\mathfrak{a}_{1} + \mathfrak{a}_{2} = \Delta$. 
We compute $\alpha_{1} \in \mathfrak{a}_{1}$ and $\alpha_{2} \in \mathfrak{a}_{2}$ such that 
$\alpha_{1} + \alpha_{2} = 1$. Then there is a short exact sequence of left $\Delta$-modules
\[
0 \longrightarrow \mathfrak{a}_1 \cap \mathfrak{a}_2 \stackrel{f}\longrightarrow \mathfrak{a}_1 \oplus \mathfrak{a}_2 \stackrel{g} \longrightarrow \Delta \longrightarrow 0
\]
with $f(a) = (a, -a)$, $g((a_1, a_2)) = a_1 + a_2$. The sequence is split by
 $s : \Delta \longrightarrow \mathfrak{a}_1 \oplus \mathfrak{a}_2$ defined by $s(1) = (\alpha_1, \alpha_2)$.
Therefore $\mathfrak{a}_{1} \oplus \mathfrak{a}_{2} = \Image(f) \oplus \Image(s)$ and so
\[
X = \Delta (\alpha_{1} x_{1} + \alpha_{2} x_{2}) \oplus (\mathfrak{a}_{1} \cap \mathfrak{a}_{2}) (x_{1} - x_{2}).
\]

We now consider the general case. 
Without loss of generality we may assume $\mathfrak{a}_1, \mathfrak{a}_2 \subseteq \Delta$. 
In order to reduce to the special case it remains to find $\tilde{\mathfrak{a}}_{2} = \mathfrak{a}_{2} \xi$ with $\xi \in D$ such that $\mathfrak{a}_{1} + \tilde{\mathfrak{a}}_{2} = \Delta$. 
We follow the proof of \cite[(27.7)]{MR1972204}.
By \cite[(27.4)]{MR1972204} we have $\fra_2 \vee \Delta$, so we can apply
the algorithmic version of Roiter's Lemma (see \S \ref{subsec:alg-roiter}). 
We choose $\alpha \in \mathfrak{a}_1 \cap \mathcal{O}_{F}$ and construct
$\varphi \colon \mathfrak{a}_2 \longrightarrow \Delta$ and an $\mathcal{O}_{F}$-torsion module $T$ such that
\[
0 \longrightarrow \mathfrak{a}_2 \stackrel{\varphi}\longrightarrow \Delta \longrightarrow T \longrightarrow 0
\]
is exact and $\alpha\mathcal{O}_{F} + \ann_{\mathcal{O}_{F}} (T) = \mathcal{O}_{F}$.
Then we claim that $\varphi(\mathfrak{a}_{2}) + \mathfrak{a}_{1} = \Delta$.
Indeed, if $1 = \rho\alpha + \beta$ with 
$\rho \in \mathcal{O}_{F}, \beta \in \ann_{\mathcal{O}_{F}} (T)$, then 
$\beta\Delta \subseteq \varphi(\mathfrak{a}_{2})$;
in particular $\beta \in \varphi(\mathfrak{a}_{2})$.

\section{Modules over maximal orders}\label{sec:mods-over-max}

Let $D$ be a skew field that is central and finite-dimensional over a number field $F$, and let $n \in \N$.

\subsection{Maximal orders up to isomorphism}

\begin{prop}[{\cite[(27.6)]{MR1972204}}]\label{prop:nice-max-order}
Let $\Delta \subseteq D$ be any maximal $\mathcal{O}_{F}$-order. 
For each right ideal $\mathfrak{a}$ of $\Delta$, 
let $\Delta' = \mathcal{O}_{l}(\mathfrak{a}) := \{ x \in D \mid x\mathfrak{a} \subseteq \mathfrak{a} \}$, and let
\[
\Lambda_{\mathfrak{a}, n} := \left(
  \begin{array}{cccc}
    \Delta & \ldots & \Delta & \mathfrak{a}^{-1} \\
    \vdots & \ddots & \vdots & \vdots \\
    \Delta & \ldots & \Delta & \mathfrak{a}^{-1} \\
    \mathfrak{a}   & \ldots & \mathfrak{a}   & \Delta'
  \end{array} \right)
\]
denote the ring of all $n \times n$ matrices $(x_{ij})$ where $x_{11}$ ranges over all elements of 
$\Delta$, \textellipsis, $x_{1n}$ ranges over all elements of $\mathfrak{a}^{-1}$, and so on.
(For $n=1$, we take $\Lambda_{\mathfrak{a}, n} := \Delta'$.) 
Then $\Lambda_{\mathfrak{a}, n}$ is a maximal $\mathcal{O}_{F}$-order in $M_{n}(D)$, and every
maximal $\mathcal{O}_{F}$-order in $M_{n}(D)$ is isomorphic to $\Lambda_{\mathfrak{a}, n}$, 
for some right ideal $\mathfrak{a}$ of $\Delta$.
\end{prop}

\subsection{Nice maximal orders}
We fix a maximal $\mathcal{O}_{F}$-order $\Delta \subseteq D$ and suppose that $n \geq 2$.
We say that a maximal $\mathcal{O}_{F}$-order $\Lambda$ in $M_{n}(D)$ is `nice' if it is equal to $\Lambda_{\mathfrak{a},n}$ for some right ideal $\mathfrak{a}$ of $\Delta$. We fix such a $\Lambda$ for the rest of this subsection.

We now give a noncommutative version of \cite[Proposition 5.3]{MR2422318}. 
In other words, we solve the problem of determining whether a left $\Lambda$-module $X$ is free of finite rank, and if so, whether generators can be computed. 
Let $e_{k,l}$ denote the matrix $(x_{i,j}) \in M_{n}(D)$ with $x_{i,j}=0$ for 
$(i,j) \neq (k,l)$ and $x_{k,l}=1$.

\begin{prop}\label{prop:gen-nice-max}
Let $X$ be a left $\Lambda$-module. Then $X$ is free of rank $d$ over $\Lambda$ if and only if
there exist $\omega_{i,j} \in X$ such that
\begin{equation}\label{eq:max-free}
e_{1,1}X 
= (\Delta\omega_{1,1} \oplus \cdots \oplus \Delta \omega_{1,n-1} \oplus \mathfrak{a}^{-1}\omega_{1,n}) \oplus \ldots \oplus 
(\Delta\omega_{d,1} \oplus \cdots \oplus \Delta \omega_{d,n-1} \oplus \mathfrak{a}^{-1}\omega_{d,n}).
\end{equation}
Further, when this is the case, $X = \Lambda \omega_{1} \oplus \cdots \oplus \Lambda \omega_{d}$
where $\omega_{j} := e_{1,1}\omega_{j,1} + \cdots + e_{n,1} \omega_{j,n}$, $j=1, \ldots, d$.
\end{prop}
\begin{proof}
Suppose that $X$ is free of rank $d$ over $\Lambda$. 
Then $e_{1,1}$ `cuts out the first row of each $\Lambda$' in $X \cong \oplus_{i=1}^{d} \Lambda$
and so $e_{1,1}X$ is of the desired form. 

Conversely, suppose that $e_{1,1}X$ is of the form given in \eqref{eq:max-free}. 
Note that since $\mathfrak{a}$ is a right $\Delta$-ideal, 
$\mathfrak{a}^{-1}$ is a left $\Delta$-ideal and by \eqref{eq:inverse-ideals} (with left and right reversed)
we have 
\[
1 \in \mathfrak{a}^{-1}\mathfrak{a} = \Delta \quad \textrm{ and } \quad 1 \in \mathfrak{a}\mathfrak{a}^{-1} = \Delta' = \mathcal{O}_{l}(\mathfrak{a}).
\]

We claim that $\Lambda\omega_{1} \oplus \cdots \oplus \Lambda\omega_{d} \subseteq X$, which reduces to showing that $\omega_{j} \in X$ for each $j$.
If $i \ne n$, then $\omega_{j,i} \in e_{1,1} X \subseteq X$, 
therefore $e_{i,1}\omega_{j,i} \in e_{i,1}X \subseteq X$. 
Furthermore, $\omega_{j,n} \in \mathfrak{a} e_{1,1} X \subseteq \mathfrak{a} X$ and hence 
$e_{n,1}\omega_{j,n} \in e_{n,1} \mathfrak{a} X \subseteq X$ 
(because $e_{n,1}\mathfrak{a} \subseteq \Lambda$).

It remains to show that $X \subseteq \Lambda\omega_{1} \oplus \cdots \oplus \Lambda\omega_{d}$. 
Note that $X = e_{1,1}X + \cdots + e_{n,n}X$.
We use the equality
$e_{1,k} \omega_j = \omega_{j, k}$ and note that for $i,j$ one has $e_{i,j}\mathfrak{a} = \mathfrak{a} e_{i,j}$. Then
\begin{eqnarray*}
e_{1,1} X 
&=& \underbrace{\Delta e_{1,1}}_{\subseteq \Lambda} \omega_{1} \oplus \cdots \oplus
               \underbrace{\Delta e_{1,n-1}}_{\subseteq \Lambda} \omega_{1} \oplus 
               \underbrace{\mathfrak{a}^{-1}e_{1,n}}_{\subseteq \Lambda} \omega_{1}
               \oplus \cdots \oplus 
               \underbrace{\Delta e_{1,n-1}}_{\subseteq \Lambda} \omega_{d}
               \oplus 
               \underbrace{\mathfrak{a}^{-1}e_{1,n}}_{\subseteq \Lambda} \omega_{d}
               \\
           &\subseteq& \Lambda\omega_{1} \oplus \cdots \oplus \Lambda\omega_{d}. 
\end{eqnarray*}
For $i \ne 1,n$ we have
\begin{eqnarray*}
  e_{i,i}X &=& e_{i,1} e_{1,1} e_{1,i} X \subseteq e_{i,1} e_{1,1} X \\
         &\subseteq& e_{i,1} (\Lambda\omega_1 \oplus \cdots \oplus \Lambda\omega_d) \subseteq
                 \Lambda\omega_1 \oplus \cdots \oplus \Lambda\omega_d.
\end{eqnarray*}
Finally we observe
\begin{eqnarray*}
  e_{n,n}X &=& e_{n,1} e_{1,1} e_{1,n} X  \subseteq e_{n,1} e_{1,1} \mathfrak{a} \underbrace{\mathfrak{a}^{-1} e_{1,n}}_{\subseteq \Lambda} X \\
         &\subseteq&  e_{n,1} e_{1,1} \mathfrak{a} X = e_{n,1} \mathfrak{a} e_{1,1} X \\
         &\subseteq& \underbrace{e_{n,1} \mathfrak{a}}_{\subseteq \Lambda}(\Lambda\omega_{1} \oplus \cdots \oplus \Lambda\omega_{d})
                 \subseteq \Lambda\omega_{1} \oplus \cdots \oplus \Lambda\omega_{d}.
\end{eqnarray*}
Therefore $X = e_{1,1}X \oplus \cdots \oplus e_{n,n}X \subseteq 
\Lambda \omega_{1} \oplus \cdots \oplus \Lambda \omega_{d}$.
\end{proof}

\subsection{Arbitrary maximal orders}
We can compute a maximal order containing a given order using \cite[Kapitel 3 and 4]{friedrichs}. However, the resulting maximal order is not necessarily nice. We address this problem by 
generalising \cite[Lemma 5.2]{MR2422318} in the following way.

We fix a maximal $\mathcal{O}_{F}$-order $\Lambda \subseteq M_{n}(D)$ and assume that it is given by an $\mathcal{O}_{F}$-pseudo-basis. 
We fix any maximal $\mathcal{O}_{F}$-order $\Delta \subseteq D$ and suppose $n \geq 2$.
We construct a right $\Delta$-lattice $N \subseteq V := D^{n}$ (column vectors) by following the proof of \cite[(21.6)]{MR1972204}. 
Let $M := \Delta^{n} \subseteq V$ and $\Lambda':= \mathcal{O}_{l}(M)$.
Then $\Lambda' = M_{n}(\Delta)$ and by Proposition \ref{prop:nice-max-order} we see that
$\Lambda'$ is a maximal $\mathcal{O}_{F}$-order in $M_{n}(D)$. 
Since $\Lambda$ and $\Lambda'$ are a pair of full $\mathcal{O}_{F}$-lattices in $M_{n}(D)$, for all but finitely many primes $\mathfrak{p}$ of $\mathcal{O}_{F}$ we have 
$\Lambda_{\mathfrak{p}} = \Lambda'_{\mathfrak{p}}$. 
The primes $\mathfrak{p}_{1}, \ldots, \mathfrak{p}_{r}$ for which 
$\Lambda_{\mathfrak{p}_{i}} \neq \Lambda'_{\mathfrak{p}_{i}}$
are determined by the generalised module index $[\Lambda : \Lambda']_{\mathcal{O}_{F}}$, 
which can be computed as follows. Compute $\mathcal{O}_{F}$-pseudo-bases so that
$\Lambda = \oplus_{i=1}^{t} \mathfrak{a}_{i} \omega_{i}$ and 
$\Lambda' = \oplus_{i=1}^{t} \mathfrak{b}_{i} \nu_{i}$ where $\mathfrak{a}_{i}, \mathfrak{b}_{i}$
are fractional $\mathcal{O}_{F}$-ideals and
find $c_{ij} \in F$ such that $\nu_{i} = \sum_{j=1}^t c_{ij} \omega_{j}$.
Then 
$[\Lambda : \Lambda']_{\mathcal{O}_{F}} = 
\det(c_{ij}) \prod_{i=1}^{t} \mathfrak{b}_{i} \mathfrak{a}_{i}^{-1}$.

For each prime $\mathfrak{p}$ of $\mathcal{O}_{F}$, we compute $u_{\mathfrak{p}} \in GL_{n}(D)$ such that 
$u_{\mathfrak{p}} \Lambda_{\mathfrak{p}}' u_{\mathfrak{p}}^{-1} =  \Lambda_{\mathfrak{p}}$, 
taking $u_{\mathfrak{p}}=1$ for $\mathfrak{p} \neq \mathfrak{p}_{1}, \ldots, \mathfrak{p}_{r}$.
To do this, we follow the proof of \cite[(17.3)(ii)]{MR1972204}, to which we refer the reader for more details. Note that $\Lambda_{\mathfrak{p}} \Lambda'_{\mathfrak{p}}$ is a full right 
$\Lambda_{\mathfrak{p}}'$-lattice in $M_{n}(D)$. 
In fact, $\Lambda_{\mathfrak{p}} \Lambda'_{\mathfrak{p}}$ is a free rank $1$ right 
$\Lambda_{\mathfrak{p}}'$-module. 
So using the algorithm given in \cite[\S 4.2]{MR2564571} we compute $u_{\mathfrak{p}}$ such that 
$\Lambda_{\mathfrak{p}} \Lambda'_{\mathfrak{p}} = u_{\mathfrak{p}} \Lambda'_{\mathfrak{p}}$.
In fact, $u_{\mathfrak{p}}$ is the element we require.

We now define $N := \cap_{\mathfrak{p}} u_{\mathfrak{p}} M_{\mathfrak{p}}$, 
giving $\Lambda = \mathcal{O}_{l}(N)$. Without loss of generality, 
we may assume that $u_{\mathfrak{p}_{i}} \in \Lambda$ for $i=1, \ldots, r$ and hence that $N \subseteq M$.
Therefore it remains to compute the finite intersection 
$N = \cap_{i=1}^{r} ( u_{\mathfrak{p}_{i}} M_{\mathfrak{p}_{i}} \cap M )$. 

We can compute each $u_{\mathfrak{p}_{i}} M_{\mathfrak{p}_{i}} \cap M$ as follows.
Let $A_{i}$ be a set of representatives of $M / u_{\mathfrak{p}_{i}} M$. 
Set $B_{i} := A_{i} \cap u_{\mathfrak{p}_{i}} M_{\mathfrak{p}_{i}}$. 
Note that for each individual element $a \in A_{i}$ one can easily check whether 
$a \in u_{\mathfrak{p}_{i}} M_{\mathfrak{p}_{i}}$. Let $C_{i}$ be a $\Z$-spanning set 
of $u_{\mathfrak{p}_{i}} M$. 
Then $B_{i} \cup C_{i}$ spans $u_{\mathfrak{p}_{i}} M_{\mathfrak{p}_{i}} \cap M$.

Hence we are reduced to computing the intersection of any two full $\mathcal{O}_{F}$-sublattices
$X,Y \subset M_{n}(D)$.  
For any full $\mathcal{O}_{F}$-sublattice $Z \subset M_{n}(D)$, we set
$Z^{*} := \{ \alpha \in M_{n}(D) \mid \mathrm{tr}(\alpha, Z) \subseteq \mathcal{O}_{F} \}$
where $\mathrm{tr}: M_{n}(D) \times M_{n}(D) \longrightarrow F$ is the bilinear reduced trace form (see \cite[p.126]{MR1972204}).
Then we have $X \cap Y = (X^{*} + Y^{*})^{*}$, which can be computed using HNF techniques over 
$\mathcal{O}_{F}$.

In summary, we have computed a lattice $N$ such that $\Lambda = \mathcal{O}_{l}(N)$.  
We now apply (the right version of) the Steinitz form algorithm of \S \ref{subsec:steinitz} to compute 
$z_{1}, \ldots, z_{n} \in V$ and a right $\Delta$-ideal $\mathfrak{a}$ such that
\[
N = z_{1}\Delta \oplus \cdots \oplus z_{n-1} \Delta \oplus z_{n}\mathfrak{a}.
\]

\begin{lemma}\label{lem:nice-form}
Let $S = (z_{1}, \ldots, z_{n}) \in GL_{n}(D)$ be the matrix with columns $z_{1}, \ldots, z_{n}$. 
Then $\Lambda = S \Lambda_{\mathfrak{a}, n} S^{-1}$.
\end{lemma}

\begin{proof}
Let $\Delta' = \mathcal{O}_{l}(\mathfrak{a})$. Then we have
\begin{eqnarray*}
\lambda \in \Lambda = \mathcal{O}_{l}(N)
&\iff&  \lambda N \subseteq N \iff \lambda S \left(
    \begin{array}{c}
      \Delta\\ \vdots \\ \Delta \\ \mathfrak{a}
    \end{array}\right) \subseteq S \left(
    \begin{array}{c}
      \Delta\\ \vdots \\ \Delta \\ \mathfrak{a}
    \end{array}\right) \\
&\iff& S^{-1} \lambda S \in \calO_l\left( \left(
    \begin{array}{c}
      \Delta\\ \vdots \\ \Delta \\ \mathfrak{a}
    \end{array}\right) \right) = 
\left(
  \begin{array}{cccc}
    \Delta & \cdots & \Delta & \mathfrak{a}^{-1} \\
    \vdots & \ddots & \vdots & \vdots \\
    \Delta & \cdots & \Delta & \mathfrak{a}^{-1} \\
    \mathfrak{a}   & \cdots & \mathfrak{a}   & \Delta'
  \end{array} \right) = \Lambda_{\mathfrak{a}, n}.
\end{eqnarray*}
\end{proof}

Hence replacing $\Lambda$ by $S^{-1}\Lambda S$ and a $\Lambda$-module $X$ by $S^{-1}X$, we may without loss of generality suppose that our maximal order is nice.

\subsection{Step (5) of Algorithm \ref{alg:the-alg}}\label{subsec:step5-outline}
Input: $\mathcal{M}_{i}$ and $\mathcal{M}_{i}X$ ($i$ fixed). We abuse notation by abbreviating 
$\mathcal{M}_{i}X$ to $X$.
\renewcommand{\labelenumi}{(\roman{enumi})}
\begin{enumerate}
\item Suppose $n=1$. 
Then $\mathcal{M}_{i}=\Delta$ for some maximal $\mathcal{O}_{F}$-order $\Delta \subseteq D$.
Use \S \ref{subsec:steinitz} to compute a Steinitz form 
$
X = \Delta b_{1} \oplus \cdots \oplus \Delta b_{d-1} \oplus \mathfrak{b} b_{d}.
$
Using (H2\textprime)(d), check whether $\mathfrak{b}$ is principal, and if so, compute $\xi \in D$ such that $\mathfrak{b}= \Delta \xi$; in this case, $b_{1}, \ldots, b_{d-1}, \xi b_{d}$ is the required $\mathcal{M}_{i}$-basis for $X$. Otherwise the algorithm terminates with the conclusion that the desired generators do not exist,
thanks to (H2\textprime)(a). 
\item We are now reduced to the case $n \geq 2$. Fix any maximal $\mathcal{O}_{F}$-order $\Delta \subseteq D$.
\item Set $\Lambda=S^{-1}\mathcal{M}_{i}S$ and replace $X$ by 
$S^{-1}X$ where $S$ is as in Lemma \ref{lem:nice-form}.
It is straightforward to see that it now suffices to determine elements 
$\omega_{1,1}, \ldots, \omega_{d,n}$ satisfying equation \eqref{eq:max-free} in Proposition 
\ref{prop:gen-nice-max}.
\item Use \S \ref{subsec:steinitz} to compute a Steinitz form 
\[ 
e_{1,1}X = \Delta b_{1} \oplus \cdots \oplus \Delta b_{dn-1} \oplus \mathfrak{b} b_{dn}.
\]
\item Again use \S \ref{subsec:steinitz} to compute a left $\Delta$-ideal $\mathfrak{c}$ and an explicit
isomorphism
\[
\varphi: \oplus_{j=1}^{d} \mathfrak{a}^{-1} \stackrel{\sim}{\longrightarrow} \oplus_{j=1}^{d-1} \Delta \oplus \mathfrak{c}.
\]
\item Using (H2\textprime)(d), check whether $\mathfrak{b} \cong \mathfrak{c}$ as left $\Delta$-ideals, and if so, compute $\xi \in D$ such that $\mathfrak{b} = \mathfrak{c} \xi$. 
Otherwise the algorithm terminates with the conclusion that the desired generators do not exist,
thanks to (H2\textprime)(a) (see Remark \ref{rmk:drop-cancellation}).
\item If a suitable $\xi \in D$ is found in the previous step then we have
  \begin{eqnarray*}
    && \Delta b_{d,(n-1)+1} \oplus \cdots \oplus \Delta b_{dn - 1} \oplus \mathfrak{b} b_{dn} \\
    &=& \Delta b_{d(n-1)+1} \oplus \cdots \oplus \Delta b_{dn - 1} \oplus \frc \xi b_{dn} \\
    &=& \mathfrak{a}^{-1} b'_{d(n-1)+1} \oplus  
    \cdots \oplus \mathfrak{a}^{-1} b'_{dn},
  \end{eqnarray*}
where the $b'_{d(n-1)+1}, \ldots,  b'_{dn}$ are computed from 
$b_{d(n-1)+1}, \ldots, b_{dn-1}, \xi b_{dn}$ using the isomorphism $\varphi$.
It is now clear how to choose the elements $\omega_{1,1}, \ldots \omega_{d,n}$.
\end{enumerate}

\begin{remark}\label{rmk:drop-cancellation}
Suppose that $\mathfrak{b}$ and $\mathfrak{c}$ are as described in steps (iv) and (v). Then
\begin{eqnarray*}
\mathfrak{b} \cong \mathfrak{c} &\implies&
\Delta^{d-1} \oplus \mathfrak{b} \cong \Delta^{d-1} \oplus \mathfrak{c} \\
&\implies&
\Delta^{dn-1} \oplus \mathfrak{b} \cong \Delta^{(n-1)d} \oplus 
\left( \oplus_{j=1}^{d} \mathfrak{a}^{-1} \right) \\
&\iff&
e_{11}X \cong \Delta^{(n-1)d} \oplus 
\left( \oplus_{j=1}^{d} \mathfrak{a}^{-1} \right) \\
&\iff& 
X \textrm{ is free over } \mathcal{M}_{i}.
\end{eqnarray*}
If $\mathfrak{b} \cong \mathfrak{c}$, then the desired generators can be computed as described above, whether or not (H2\textprime)(a) holds. The purpose of (H2\textprime)(a) is to ensure that the first two implication arrows above are in fact equivalences; so if $\mathfrak{b} \ncong \mathfrak{c}$, then the algorithm terminates in step (iv) with the conclusion that the desired generators do not exist. 
However, if (H2\textprime)(a) does not hold and $\mathfrak{b} \ncong \mathfrak{c}$, then generators may or may not exist.
\end{remark}

\begin{remark}
If $nd=1$, then we are immediately reduced to solving the principal ideal problem (i.e.\ applying (H2\textprime)(d)) in step (i), so there is no need for $\Delta$ to have locally free cancellation; this is the reason for `if $nd>1$' in (H2\textprime)(a). Also note that a simplified version of Remark \ref{rmk:drop-cancellation} applies to step (i) when $n=1$ and $d>1$.
\end{remark}

\section{Enumerating Units}\label{sec:enunits}

Let $d,n \in \N$ and let $F$ be a number field. 
Let $D$ be a skew field with centre $F$ and let
$\Lambda$ be some maximal $\mathcal{O}_{F}$-order in $M_{n}(D)$. 
If $n=1$, then $\Lambda=\Delta$ for some maximal $\mathcal{O}_{F}$-order $\Delta$ of $D$.
If $n \geq 2$, we may choose a maximal $\mathcal{O}_{F}$-order $\Delta$ of $D$ and
by Lemma \ref{lem:nice-form} we may assume that $\Lambda$ is nice, i.e.\ of the form
$\Lambda_{\mathfrak{a}, n}$ for some right $\Delta$-ideal $\mathfrak{a}$. 
Let $\mathfrak{g}$ be some non-zero ideal of $\mathcal{O}_{F}$ and set $\overline{\Delta} := \Delta / \frg\Delta$. Let $\mathfrak{f} := \mathfrak{g} \Lambda$ and set 
$\overline{\Lambda} := \Lambda/\mathfrak{f}$. 
Throughout this section, we identify $M_{d}(\Lambda)$ with a subring of $M_{dn}(D)$ in the obvious way. 
We wish to compute a set of representatives $U \subset GL_{d}(\Lambda)$ of the image of
the natural projection map 
$\pi: GL_{d}(\Lambda) \longrightarrow GL_{d}(\overline{\Lambda})$, thereby generalising the results
of \cite[\S 6]{MR2422318}.

\subsection{A reduction step}\label{subsec:reduction-step}

As in the last paragraph of \S \ref{subsec:steinitz}, we compute $\xi \in D$ and a right $\Delta$-ideal 
$\mathfrak{b}$ such that $\mathfrak{a} = \xi \mathfrak{b}$ and $\mathfrak{b} + \frg \Delta = \Delta$.
By a special case of the noncommutative extended Euclidean algorithm given  in \S \ref{subsec:ex-euclid}, we can find $b \in \mathfrak{b}$ and $y \in \frg \Delta$ such that $b + y = 1$. 
We define diagonal matrices
\[
\Phi_{1} := \left(
  \begin{array}{cccc}
    1 & & & \\  & \ddots & &  \\ & & 1 &  \\ & & & \xi^{-1} \\
  \end{array} \right) 
\quad \textrm{ and } \quad
\Phi_{2} := \left(
  \begin{array}{cccc}
    1 & & &  \\ & \ddots & &  \\ & & 1 &  \\ & & & \xi b \\
  \end{array} \right)
\]
in $GL_{n}(D)$. Then we have homomorphisms
\begin{eqnarray*}
f_1 \colon  GL_{nd}(\Delta) & \longrightarrow & GL_d(\Lambda), \\
   A  = \left( A_{ij} \right)_{1 \le i,j \le d} & \mapsto & \left( \Phi_2 A_{ij} \Phi_1 \right)_{1 \le i,j \le d}
\end{eqnarray*}
and 
\begin{eqnarray*}
  f_2 \colon GL_{d}(\Lambda) & \longrightarrow & GL_{nd}(\Delta), \\
   B  = \left( B_{ij} \right)_{1 \le i,j \le d}            & \mapsto & \left( \Phi_1 B_{ij} \Phi_2 \right)_{1 \le i,j \le d},
\end{eqnarray*}
where $A_{ij} \in M_{n}(\Delta)$ and $B_{ij} \in \Lambda$.
Note that the induced map $\bar f_{1} : GL_{nd}(\overline{\Delta}) \longrightarrow 
GL_{d}(\overline{\Lambda})$ 
is an
isomorphism with inverse $\bar f_{2}$. In summary, we have a commutative diagram
\[
\xymatrix{
GL_{nd}(\Delta)  \ar@<2pt>[r]^{f_1} \ar[d] & GL_d(\Lambda) \ar@<2pt>[l]^{f_2} \ar[d]^\pi \\
GL_{nd}(\overline{\Delta})  \ar@<2pt>[r]^{\bar f_1} & GL_d(\overline{\Lambda}) \ar@<2pt>[l]^{\bar f_2}
}
\]
where the lower horizontal arrows are isomorphisms. 
We set $k := dn$ and conclude that it suffices to compute a set
of representatives $U \subseteq GL_{k}(\Delta)$ of the image of the natural projection map 
$GL_{k}(\Delta) \longrightarrow GL_{k}(\overline{\Delta})$, which by abuse of notation we also denote by $\pi$.

\subsection{Computing a set of representatives of the map 
$\pi: GL_{k}(\Delta) \longrightarrow GL_{k}(\overline{\Delta})$}

We assume that $k>1$ and deal with the case $k=1$ in \S \ref{subsec:casek=1}.

We first recall some definitions from algebraic $K$-theory and refer the reader to 
\cite[\S 40]{MR892316} for more details. 
Let $R$ be a unital ring and let $m \in \N$. For $x \in R$ and $i,j \in \{1, \ldots m \}$ with $i \neq j$,
the elementary matrix $E_{ij}(x)$ is the matrix in $GL_{m}(R)$ that has $1$ in every diagonal entry, 
has $x$ in the $(i,j)$-entry and is zero elsewhere. 
Let $E_{m}(R)$ denote the subgroup of $GL_{m}(R)$ generated by all elementary matrices.
Let $E(R)$ and $GL(R)$ be the direct limits given by the obvious inclusions 
$E_{m}(R) \longrightarrow E_{m+1}(R)$ and $GL_{m}(R) \longrightarrow GL_{m+1}(R)$. 
Then $K_{1}(R)$ is defined to be the abelian group $GL(R)/E(R)$.

\begin{lemma}\label{lem:equal-projs}
We have 
$E_{k}(\overline{\Delta})=\pi(E_{k}(\Delta))=\pi(GL_{k}(\Delta) \cap E(\Delta))$.
\end{lemma}

\begin{proof}
The first equality is clear.
The quotient ring $\overline{\Delta}$ is semilocal and so has stable range $1$ 
by \cite[(40.31)]{MR892316} (see \cite[(40.39)]{MR892316} for the definition of stable range).
Then by the Injective Stability Theorem (see \cite[(40.44)]{MR892316}), we have 
$E(\overline{\Delta}) \cap GL_{k}(\overline{\Delta}) = E_{k}(\overline{\Delta})$.

We consider $\pi$ to be the restriction of the natural projection map 
$GL(\Delta) \longrightarrow GL(\overline{\Delta})$, which we also denote by $\pi$. 
Then in $GL(\overline{\Delta})$ we have
\begin{eqnarray*}
\pi(E(\Delta) \cap GL_{k}(\Delta)) &\subseteq& \pi(E(\Delta)) \cap \pi(GL_{k}(\Delta)) \\
&=& E(\overline{\Delta}) \cap \pi(GL_{k}(\Delta)) \\
&\subseteq&  E(\overline{\Delta}) \cap GL_{k}(\overline{\Delta}) \\
&=& E_{k}(\overline{\Delta}) = \pi(E_{k}(\Delta)).
\end{eqnarray*}
However, it is clear that $E_{k}(\Delta) \subseteq GL_{k}(\Delta) \cap E(\Delta)$
and so $\pi(E_{k}(\Delta)) \subseteq \pi(GL_{k}(\Delta) \cap E(\Delta))$. 
As we have shown the reverse inclusion above, the desired equality now follows.
\end{proof}

\begin{lemma}\label{lemma:GL2-reps}
Let $U'$ be a set of representatives of the map $GL_{2}(\Delta) \longrightarrow K_{1}(\Delta)$. 
Then $\pi(GL_{k}(\Delta))$ is generated by $E_{k}(\overline{\Delta})$ and $\pi(U')$.
\end{lemma}

\begin{proof}
The map $GL_{k}(\Delta) \longrightarrow K_{1}(\Delta)$ is surjective by \cite[(41.23)]{MR892316} and so
\[
GL_{k}(\Delta)/(GL_{k}(\Delta) \cap E(\Delta)) \cong K_{1}(\Delta).
\]
Hence $GL_{k}(\Delta)$ is generated by $U'$ and $E(\Delta) \cap GL_{k}(\Delta)$, and so the result now follows from Lemma \ref{lem:equal-projs}.
\end{proof}

Let $\nr: GL_{k}(\Delta) \longrightarrow \mathcal{O}_{F}^{\times}$ 
denote the reduced norm map as defined in \cite[\S 7D]{MR632548} and write 
$\nr: K_{1}(\Delta) \longrightarrow \mathcal{O}_{F}^{\times}$ for the induced map. Then define
\begin{equation}\label{eqn:SL-SK-def}
SL_{k}(\Delta) := \{  x \in GL_{k}(\Delta) : \mathrm{nr}(x) = 1 \} \quad \textrm{ and } \quad
SK_{1}(\Delta) := \{  x \in K_{1}(\Delta) : \mathrm{nr}(x)=1 \}.
\end{equation}
Since $k=nd>1$, by (H2\textprime)(b) we have $\nr(\Delta^{\times})=\mathcal{O}_{F}^{\times +}$. 
Hence by (H2\textprime)(c) we can compute a set $V$ of representatives of the map
$\nr:\Delta^{\times} \longrightarrow \mathcal{O}_{F}^{\times +}$. 
Let $U$ be a set of representatives of the map $SL_{2}(\Delta) \longrightarrow SK_{1}(\Delta)$.
(We shall see how to compute $U$ in \S \ref{subsec:SK1-reps}.)

\begin{prop}\label{prop:final-proj-gen}
Assuming \emph{(H2\textprime)(b)}, $\pi(GL_{k}(\Delta))$ is generated by $E_{k}(\overline{\Delta})$, $\pi(V)$ and $\pi(U)$. 
(We consider $\Delta^{\times} = GL_{1}(\Delta)$ as a subgroup of $GL_{k}(\Delta)$ in the natural way.)
\end{prop}

\begin{proof}
By Lemma \ref{lemma:GL2-reps} we are reduced to showing that 
$\pi(U') \subseteq \langle \pi(V), \pi(U) \rangle$ where $U'$ is a set of representatives of
the map $GL_{2}(\Delta) \longrightarrow K_{1}(\Delta)$. 
By \cite[(45.15)]{MR892316} we have a short exact sequence
\[
1 \longrightarrow SK_{1}(\Delta) \longrightarrow K_{1}(\Delta) \stackrel{\mathrm{nr}}{\longrightarrow} \mathcal{O}_{F}^{\times +} \longrightarrow 1.
\]
However, the map $\nr : \Delta^{\times} \longrightarrow \mathcal{O}_{F}^{\times +}$ factors via 
$K_{1}(\Delta)$ and is surjective, and so the 
desired result now follows.
\end{proof}

\subsection{Computing a set of representatives of the map $SL_{2}(\Delta) \longrightarrow SK_{1}(\Delta)$}\label{subsec:SK1-reps}

We first recall that the map $GL_{2}(\Delta) \longrightarrow K_{1}(\Delta)$ is surjective by \cite[(41.23)]{MR892316};
hence by the definitions given in \eqref{eqn:SL-SK-def}, the map 
$SL_{2}(\Delta) \longrightarrow SK_{1}(\Delta)$ is also surjective .

Let $m$ denote the index of $D$, i.e.\ $[D : F] = m^{2}$.
For any prime $\mathfrak{p}$ of $F$, let $F_{\mathfrak{p}}$ denote the $\mathfrak{p}$-adic completion of $F$ and define $D_{\mathfrak{p}} := F_{\mathfrak{p}} \otimes_{F} D$. 
Following \cite[(45.14)]{MR892316}, we may write
\[
D_{\mathfrak{p}} \cong M_{\kappa_{\mathfrak{p}}}(\Omega_{\mathfrak{p}}) 
\quad \textrm{ and } \quad
[\Omega_{\mathfrak{p}} : F_{\mathfrak{p}}] = m_{\mathfrak{p}}^2,
\]
where $\Omega_{\mathfrak{p}}$ is a skew field with centre $F_{\mathfrak{p}}$ and index 
$m_{\mathfrak{p}}$. For each $\mathfrak{p}$, we have $m=\kappa_{\mathfrak{p}}m_{\mathfrak{p}}$,
so $m_{\mathfrak{p}}$ divides $m$.
We say that $\mathfrak{p}$ is ramified in $D$ if $m_{\mathfrak{p}}>1$, and unramified if 
$m_{\mathfrak{p}}=1$. (Note that the definition of ramification here depends crucially on the fact that $F$ is the centre of $D$; in what follows below, the notion of ramification in finite extensions of number fields is the usual one.)
Let $S_{ram}$ be the set of finite primes of $F$ that ramify in $D$, and note that
this is a finite set by \cite[(32.1)]{MR1972204}.
 
\begin{lemma}\label{lemma:W-field}
There exists a number field $W$ such that
\begin{itemize}
\item[(a)] $F \subseteq W \subseteq D$;
\item[(b)] $W/F$ is cyclic of degree $m$; and
\item[(c)] all primes in $S_{ram}$ are inert (i.e. unramified and non-split) in $W/F$.
\end{itemize}
\end{lemma}

\begin{proof}
We wish to apply Grunwald-Wang Theorem (see \cite[Theorem 9.2.8]{MR2392026}, for example).
However, we first have to check that we are not in the `special case'.
Suppose for a contradiction that we are in the `special case'; then 
by \cite[top of p.528]{MR2392026}, in particular we have
$m=2^{r}m'$ where $m'$ is odd and $r \geq 3$ and $F(\zeta_{2^{r}})/F$ is \emph{not} cyclic.
However, by the Benard-Schacher Theorem (see \cite{MR0302747} or \cite[(74.20)]{MR892316}) 
we must have $\zeta_{m} \in F$ and so $F(\zeta_{2^{r}})=F$, giving a contradiction.
Hence we may apply the Grunwald-Wang Theorem to show that 
there exists a cyclic extension $W/F$ of degree $m$
in which all primes in $S_{ram}$ are inert, and every real prime is ramified
(so $W$ is totally imaginary).
Now \cite[(32.15)]{MR1972204} shows that $W$ is a splitting field for $D$ and
\cite[(28.10)]{MR1972204} shows that $W$ can be embedded in $D$.
\end{proof}

\begin{remark}\label{rmk:algorithmic-Grunwald-Wang}
If $D$ is a quaternion algebra (i.e.\ $m=2$) then the required field $W$ is a quadratic extension of $F$ and can be found easily in practice - this has been done for all Wedderburn components of group rings $\Q[G]$ where $G$ is a generalised quaternion group with $|G| < 48$ (see the sample file). 
In the general case, one can employ the algorithmic `weak' Grunwald-Wang Theorem of
\cite[Algorithm 13] {MR2531221}.  As stated, this algorithm does not control ramification at primes above $2$, but this is only a problem in the aforementioned `special case'; as shown in the proof of Lemma \ref{lemma:W-field}, this case does not occur in the situation of interest to us.
\end{remark}

Let $\mathcal{T}$ be the category of finitely generated $\mathcal{O}_{F}$-torsion $\Delta$-modules.
For each finite prime $\mathfrak{p}$ of $\mathcal{O}_{F}$, set 
$\Delta_{\mathfrak{p}} := \mathcal{O}_{F_{\mathfrak{p}}} \otimes_{\mathcal{O}_{F}} \Delta$ 
(so $\Delta_{\mathfrak{p}}$ is a maximal $\mathcal{O}_{F_{\mathfrak{p}}}$-order in 
$D_{\mathfrak{p}}$), and let $\mathcal{T}_{\mathfrak{p}}$ be the category of finitely generated $\mathfrak{p}$-torsion $\Delta$-modules. 
Let $\varepsilon_{\mathfrak{p}}: K_{1}(\mathcal{T}_{\mathfrak{p}}) \longrightarrow K_{1}(\Delta_{\mathfrak{p}})$ be the map defined in the proof of \cite[(45.13)]{MR892316},
where $\varepsilon_{\mathfrak{p}}$ is denoted $\varepsilon$ and it is shown that 
$SK_{1}(\Delta_{\mathfrak{p}}) = \Image(\varepsilon_{\mathfrak{p}}$).
Define $\varepsilon: K_{1}(\mathcal{T}) \longrightarrow K_{1}(\Delta)$ analogously to 
$\varepsilon_{\mathfrak{p}}$; in the proof of \cite[(45.15)]{MR892316} $\varepsilon$ 
is unlabelled and it is shown that $SK_{1}(\Delta) = \Image(\varepsilon)$.
Furthermore, there is a canonical isomorphism 
$K_{1}(\mathcal{T}) \cong \coprod_{\mathfrak{p}} K_{1}(T_{\mathfrak{p}})$.
Therefore we have a commutative diagram
\[
\xymatrix{
\coprod K_{1}(\mathcal{T}_{\mathfrak{p}})  \ar[r]^{\cong}  \ar[d]^{\coprod \varepsilon_{\mathfrak{p}}} 
&  K_{1}(\mathcal{T}) \ar[d]^{\varepsilon} \\
\coprod SK_{1}(\Delta_{\mathfrak{p}})  \ar[r]^{\cong} & SK_{1}(\Delta),
}
\]
where the vertical maps are surjective.

Write $S_{ram} = \{ \mathfrak{p}_{1}, \ldots, \mathfrak{p}_{s} \}$.
To ease notation, we abbreviate $m_{\mathfrak{p}_{i}}$ to $m_{i}$, etc.
By \cite[(45.15)]{MR892316}, we in fact have that
$SK_{1}(\Delta) \cong \coprod_{i=1}^{s} SK_{1}(\Delta_{\mathfrak{p}_{i}})$ 
and each
$SK_{1}(\Delta_{\frp_i})$ is cyclic of order $(q_{i}^{m_{i}} - 1) / (q_{i} -1)$,
where $q_{i} := | \mathcal{O}_{F} / \mathfrak{p}_{i} |$. 
It is also shown that $K_{1}(\mathcal{T}_{\mathfrak{p}_{i}})$ is cyclic of order $q^{m_{i}}-1$.
Hence our strategy shall be as follows: for each $i$, find an element of 
$K_{1}(\mathcal{T}_{\mathfrak{p}_{i}})$ that maps to a generator of 
$SK_{1}(\Delta_{\mathfrak{p}_{i}})$, and compute a representative in $SL_{2}(\Delta)$
of this generator.

We now fix $i \in \{ 1, \ldots, s \}$.
Let $W_{i} \subseteq W$ denote the unique subfield with $[W_{i} : F] = m_i$ and 
write $\sigma_{i} \in \Gal(W/F)$ for the Frobenius substitution associated to $\mathfrak{p}_{i}$.
Write $\mathfrak{P}_{i}$ for the unique prime of $W$ above $\mathfrak{p}_{i}$ and 
$\xi_{i} \in \mathcal{O}_{W} / \mathfrak{P}_{i}$ for a primitive root of unity of order $q_{i}^{m_{i}} - 1$. 
If $\hat{\mathfrak{p}}_{i}$ denotes the unique prime of $W_{i}$ lying over $\mathfrak{p}_{i}$, 
then we can in fact choose $\xi_{i} \in \mathcal{O}_{W_{i}} / \hat{\mathfrak{p}}_{i}$.

By the Skolem-Noether Theorem (see \cite[(7.21)]{MR1972204}) there exists $\alpha_{i} \in \Delta$ such that
$\beta^{\sigma_{i}} = \alpha_{i} \beta \alpha_{i}^{-1}$ for all $\beta \in W$.
Such an element $\alpha_{i}$ can be computed as follows.
Fix an $F$-basis of $\beta_{1}, \ldots, \beta_{m}$ of $W$ and an $F$-basis
$\omega_{1}, \ldots, \omega_{m^2}$ of $D$. 
Write $\alpha_{i} = x_{1} \omega_{1} + \cdots + x_{m^{2}}\omega_{m^{2}}$ where 
$x_{1}, \ldots, x_{m^{2}}$ are elements of $F$ to be determined.
Then $\beta^{\sigma_{i}} = \alpha_{i} \beta \alpha_{i}^{-1}$ for all $\beta \in W$ 
is equivalent to  $\alpha_{i} \beta_{j}^{\sigma_{i}} = \alpha_{i} \beta_{j}$ for
$j = 1, \ldots, m$. 
Hence we have a system of linear equations for $x_{1}, \ldots, x_{m^{2}}$, which can be easily solved using standard algorithms. 
Once a solution is found, it only remains to clear denominators to ensure that $\alpha_{i} \in \Delta$.

The following construction is inspired by the proof of \cite[(45.13)]{MR892316}.
Choose any non-zero $\rho_{i} \in \alpha_{i} \Delta \cap W$ and
let $\{ \mathfrak{Q}_{1}, \ldots, \mathfrak{Q}_{t} \}$ 
be the union of $\{ \mathfrak{Q} \text{ divides } (\rho_{i}) \}$ and $\{ \mathfrak{P}_{i} \}$.
Assume that $\mathfrak{Q}_{1} = \mathfrak{P}_{i}$ and apply the Chinese Remainder Theorem
(one can use the $\texttt{CRT}$ function of Magma; 
also see \cite[Proposition 1.3.11]{MR1728313})
to compute $\eta_{i} \in \mathcal{O}_{W}$ such that
\[
\eta_{i} \equiv \xi_{i}^{\sigma_{i}} \pmod{\mathfrak{P}_{i}}
\quad \textrm{ and } \quad
\eta_{i} \equiv 1 \pmod{\mathfrak{Q}_{j}} \quad \textrm{ for } \quad j = 2, \ldots, t.
\]
We set $\omega_{i} := \eta_{i}^{\sigma_{i}^{-1}}$.
Then we have
\[
\omega_{i}^{\sigma_{i}} = \eta_{i} \equiv \xi_{i}^{\sigma_{i}} \equiv \xi_{i}^{q_{i}} \pmod{\mathfrak{P}_{i}}, \quad
\omega_{i} \equiv \xi_{i} \pmod{\mathfrak{P}_{i}}, \quad
\textrm{ and } \quad
(\omega_{i}^{\sigma_{i}}, \rho_{i}) = \mathcal{O}_{W}.
\]

Now for each $i \in \{1, \ldots, s\}$ we have 
$\omega_{i}^{\sigma_{i}} \alpha_{i} = \alpha_{i} \omega_{i}$ by definition of $\alpha_{i}$.
Hence we have a commutative diagram with exact rows
\[
\xymatrix{
0 \ar[r] & \Delta_{\frp_i} \ar[r]^{\alpha_i}  \ar[d]^{\omega_i} & 
\Delta_{\frp_i} \ar[r]  \ar[d]^{\omega^{\sigma_i}_i} & 
\Delta_{\frp_i} /\alpha_i\Delta_{\frp_i} \ar[r] \ar[d]^{\tau_i} & 0 \\
0 \ar[r] & \Delta_{\frp_i} \ar[r]^{\alpha_i}  & \Delta_{\frp_i} \ar[r]  & 
\Delta_{\frp_i} /\alpha_i\Delta_{\frp_i} \ar[r] & 0, \\
}
\]
where the maps in the left-hand square are induced by left multiplication and 
the map $\tau_{i}$ is induced by the middle vertical map. 
Note that all vertical maps are isomorphisms.
It follows that 
$[\Delta_{\mathfrak{p}_{i}} / \alpha_{i} \Delta_{\mathfrak{p}_{i}}, \tau_{i}]
\in K_{1}(\mathcal{T}_{\mathfrak{p}_{i}})$ is mapped by 
$\varepsilon_{\mathfrak{p}_{i}}$ to 
$[\Delta_{\mathfrak{p}_{i}},\omega_{i}^{\sigma_{i} - 1}] \in SK_{1}(\Delta_{\mathfrak{p}_{i}})$.
There is a natural homomorphism 
$K_{1}(\Delta_{\mathfrak{p}_{i}}) \longrightarrow
 K_{1}(\Delta_{\mathfrak{p}_{i}} / \mathrm{rad} \Delta_{\mathfrak{p}_{i}}) \cong
\left( \mathcal{O}_{W_{i}} / \hat{\mathfrak{p}}_{i} \right)^{\times} = \langle \xi_{i} \rangle$,
under which $[\Delta_{\mathfrak{p}_{i}},\omega_{i}^{\sigma_{i} - 1}]$ maps to $\xi_{i}^{q_{i}-1}$.
However, both $SK_{1}(\Delta_{\mathfrak{p}_{i}})$ and $\langle \xi_{i}^{q_{i}-1} \rangle$
are cyclic of order $(q_{i}^{m_{i}} - 1) / (q_{i} -1)$.
Therefore 
$\varepsilon_{\mathfrak{p}_{i}}([\Delta_{\mathfrak{p}_{i}} / \alpha_{i} \Delta_{\mathfrak{p}_{i}}, \tau_{i}]) = [\Delta_{\mathfrak{p}_{i}},\omega_{i}^{\sigma_{i} - 1}]$
generates $SK_{1}(\Delta_{\mathfrak{p}_{i}})$.

It remains to compute a representative in $SL_{2}(\Delta)$ of the image of
\[
\left[ \Delta_{\frp_i} / \alpha_i \Delta_{\frp_i}, \tau_i \right] \in K_{1}(\calT_{\frp_i})
\subseteq \coprod K_{1}(\calT_{\frp}) \cong K_{1}(\calT)
\]
under the map $\varepsilon: K_{1}(\mathcal{T}) \longrightarrow K_{1}(\Delta)$. To that end, 
it suffices to construct a commutative diagram of the form
\[
\xymatrix{
0 \ar[r] & \Delta^{2} \ar[r]  \ar[d] & 
\Delta^{2} \ar[r]  \ar[d] & 
\Delta_{\mathfrak{p}_{i}} /\alpha_{i}\Delta_{\mathfrak{p}_{i}} \ar[r] \ar[d]^{\tau_{i}} & 0 \\
0 \ar[r] & \Delta^{2} \ar[r]  & \Delta^{2} \ar[r]  & 
\Delta_{\mathfrak{p}_{i}} /\alpha_{i}\Delta_{\mathfrak{p}_{i}} \ar[r] & 0 \\
}
\]
where all vertical maps are isomorphisms and the rows are exact. 
For this it is enough to construct the middle vertical isomorphism.

Compute $\beta, \eta \in \mathcal{O}_{W}$ such that $\beta\omega_{i}^{\sigma_{i}} - \eta\rho_{i} = 1$
(use  \cite[Algorithm 1.3.2]{MR1728313}) and consider the commutative diagram
\[
\xymatrix{
0 \ar[r] & \Delta^{2} \ar[rr]^{
  \Tiny \left(\begin{array}{cc}
    \alpha_{i} & 0 \\ 0 & 1
  \end{array}\right)} \ar[d]^{S_{i}} & &
\Delta^{2} \ar[rr]^{(1 \, 0) \quad}  \ar[d]^{T_{i}} & &
\Delta_{\mathfrak{p}_{i}} /\alpha_{i}\Delta_{\mathfrak{p}_{i}} \ar[r] \ar[d]^{\tau_{i}} & 0 \\
0 \ar[r] & \Delta^{2} \ar[rr]_{\Tiny\left(\begin{array}{cc}
    \alpha_{i} & 0 \\ 0 & 1
  \end{array}\right)}  & & \Delta^{2} \ar[rr]_{(1 \, 0) \quad}  & &
\Delta_{\mathfrak{p}_{i}} /\alpha_{i}\Delta_{\mathfrak{p}_{i}} \ar[r] & 0,
}
\]
with matrices
\[
S_{i} := \left(
  \begin{array}{cc}
    \omega_i & \alpha_i^{-1}\rho_i\\ \eta\alpha_i & \beta
  \end{array} \right)
\quad \textrm{ and } \quad
T_{i} := \left(
  \begin{array}{cc}
    \omega_i^{\sigma_i} & \rho_i\\ \eta & \beta
  \end{array} \right).
\]
Note that $S_{i} \in M_{2}(\Delta)$ by our choice of $\rho_{i}$ and that the middle vertical map is an
isomorphism because $T_{i} \in GL_{2}(\mathcal{O}_{W})$ by our choice of $\beta$ and $\eta$. 
It follows that
\[
\varepsilon\left( \left[ \Delta_{\frp_i} / \alpha_i \Delta_{\frp_i}, \tau_i \right] \right) 
= \left[ \Delta^{2}, S_{i}^{-1}T_{i} \right].
\]
Hence $SK_1(\Delta)$ is generated by the classes 
$[ \Delta^{2}, S_{i}^{-1}T_{i}]$ for $i = 1, \ldots, s$, and so it is now straightforward 
to compute a set respresentatives of $SL_{2}(\Delta) \longrightarrow SK_{1}(\Delta)$.

\subsection{The case $k=1$}\label{subsec:casek=1}
We have to compute a set of  representatives of the natural projection map 
$\pi: \Delta^{\times} \longrightarrow \overline{\Delta}^{\times}$. 
If we can compute a set of generators of $\Delta^{\times}$,
it is straightforward to compute a set of representatives of $\pi$ 
(see \S \ref{subsec:norm-unit-reps}). 
Otherwise, we can and do assume that $\nr(\Delta^{\times}) = \mathcal{O}_{F}^{\times +}$
by (H2\textprime)(c) and proceed as follows.

We can compute a set of representatives $U$ of the natural map $\theta: SL_{2}(\Delta) \longrightarrow SK_{1}(\Delta)$ by the method of \S \ref{subsec:SK1-reps}. 
Furthermore, by (H2\textprime)(c), we can compute a set of representatives $V$ of the reduced norm map $\nr : \Delta^{\times} \longrightarrow \mathcal{O}_{F}^{\times +}$.
Let $\Delta_{1}^{\times}$ denote the kernel of this map. Let $\pi: GL_{2}(\Delta) \longrightarrow GL_{2}(\overline{\Delta})$ and $\iota:\Delta^{\times} \longrightarrow K_{1}(\Delta)$ denote the natural maps. 

Let $w \in \Delta^{\times}$. Then there exists an element $v \in V$ such that $\nr(v) = \nr(w)$; 
hence $a:=wv^{-1} \in \Delta_{1}^{\times}$ and so 
$\iota(a) \in \iota(\Delta_{1}^{\times}) \subseteq SK_{1}(\Delta)$. 
Now there exists $u \in U$ such that $\theta(u) = \iota(a)$ in $SK_{1}(\Delta)$. 
Then we have
\begin{eqnarray*}
u &\equiv& \left(
    \begin{array}{cc}
      a&0\\0&1
    \end{array}\right) \pmod{GL_{2}(\Delta) \cap E(\Delta)},\\
\text{so } \quad \pi(u) &\equiv& \left(
    \begin{array}{cc}
      \bar a&0\\0&1
    \end{array}\right) \pmod{E_{2}(\overline{\Delta})} \quad \text{ by Lemma \ref{lem:equal-projs},}\\
\text{and hence}&& \left(
    \begin{array}{cc}
      \bar a&0\\0&1
    \end{array}\right) \in S := \langle \pi(U), E_2(\overline{\Delta})) \rangle.
\end{eqnarray*}
It is straightforward to compute the finite set $S$.
Let $\overline{V}'$ denote the set of elements in $S$ which are of the form
$ \left(
    \begin{array}{cc}
      \bar a&0\\0&1
    \end{array}\right)$. Then $\pi(\Delta^{\times})$ is generated by $\pi(V)$ and $\overline{V}'$.

\subsection{Step (7) of Algorithm \ref{alg:the-alg}}

\renewcommand{\labelenumi}{(\roman{enumi})} 
Input: $d,n \in \N$; $D$ a skew field that is central and finite-dimensional over a number field $F$; 
$\mathfrak{g}$ a non-zero ideal of  $\mathcal{O}_{F}$; $\Delta$ a maximal $\mathcal{O}_{F}$-order in $D$; 
$\Lambda=\Lambda_{\mathfrak{a},n}$ for some right $\Delta$-ideal $\mathfrak{a}$, a nice maximal 
$\mathcal{O}_{F}$-order in $M_{n}(D)$.
\begin{enumerate}
\item Set $\overline{\Delta}:=\Delta/\mathfrak{g}\Delta$ and 
$\overline{\Lambda}:=\Lambda/\mathfrak{g}\Lambda$.
\item Suppose $\nr(\Delta^{\times}) \neq \mathcal{O}_{F}^{\times +}$. 
Then $nd=1$ by (H2\textprime)(b), and by (H2\textprime)(c) we can compute generators for $\Delta^{\times}$. 
It is then straightforward to compute a set of representatives for $\pi: \Delta^{\times} \longrightarrow \overline{\Delta}^{\times}$. 
\item We are now reduced to the case $\nr(\Delta^{\times})=\mathcal{O}_{F}^{\times +}$.
By (H2\textprime)(c) we can compute a set of representatives $V$ of 
$\nr:\Delta^{\times} \longrightarrow \mathcal{O}_{F}^{\times +}$.
By \S \ref{subsec:SK1-reps} we can also compute a set of representatives $U$ of 
$SL_{2}(\Delta) \longrightarrow SK_{1}(\Delta)$.
\item If $nd=1$, then we proceed by the method of \S \ref{subsec:casek=1}.
\item It remains to consider the case $nd>1$. 
By \S \ref{subsec:reduction-step} we are reduced to computing a set of representatives of
$\pi : GL_{nd}(\Delta) \longrightarrow GL_{nd}(\overline{\Delta})$.
\item $E_{nd}(\overline{\Delta})$ is generated by the elementary matrices $E_{ij}(b_{ijk})$
for $i,j \in \{ 1, \ldots, nd \}$, $i \neq j$, where for fixed $i,j$, $\{ b_{ijk} \}$ is a $\Z$-spanning
set for $\overline{\Delta}$.
\item By Proposition \ref{prop:final-proj-gen}, we now have an explicit generating set for 
$\pi(GL_{nd}(\Delta))$, and so it is now straightforward to compute the desired set of representatives.
\end{enumerate}

\section{Implementation and computational results}\label{sec:comp-results}

Let $L/K$ be a finite Galois extension of number fields with Galois group $G$. 
Let $E$ be a subfield of $K$ and set $d:=[K:E]$. 
As discussed in the introduction, Algorithm \ref{alg:the-alg} can be applied in the
situation $X=\mathcal{O}_{L}$ and $\mathcal{A}=\mathcal{A}(E[G];\mathcal{O}_{L})$.
This is implemented in Magma (\cite{MR1484478}) for certain groups $G$ in the case $K=E=\Q$.
The source code, instructions, and input files are available from
\[
\texttt{http://www.mathematik.uni-kassel.de/$\sim$bley} .
\]

The cases in which $G$ is abelian, dihedral, or $G=A_{4}, S_{4}$ were already implemented based
on the special case of Algorithm \ref{alg:the-alg} presented in \cite{MR2422318}. In the case of 
$G=S_{4}$, the method of \cite[\S 7]{MR2422318} was used to speed up the enumeration.
The more general version of Algorithm \ref{alg:the-alg} is now also implemented for $G=Q_{4n}$ (the generalised quaternion group of order $4n$) and $G=Q_{8} \times C_{2}$. 
In all cases, the running time is reasonable for $|G| \leq 16$.

Assuming that $\mathcal{O}_{L}$ is locally free over $\mathcal{A}$, 
the class group methods described in \cite{MR2564571} allow us to compute the class of 
$\mathcal{O}_{L}$ in the locally free class group $\mathrm{cl}(\mathcal{A})$; in particular, 
we are able to determine whether $\mathcal{O}_{L}$ is stably free over $\mathcal{A}$.
In the case that $\mathcal{A}$ has locally free cancellation (see \S \ref{subsec:loc-free-cancel})
stably free is equivalent to free and we are therefore able to use the class group methods to check the correctness of our implementation of Algorithm \ref{alg:the-alg}: we must eventually 
find a generator for $\mathcal{O}_{L}$ over $\mathcal{A}$ if and only if the class of $\mathcal{O}_{L}$ 
is trivial in $\mathrm{cl}(\mathcal{A})$. (Note that $\mathcal{A}$ having locally free cancellation is in general different from (H2\textprime)(a).)
Once a generator has been computed, it is easy to verify the correctness of the computation.

The class group methods of \cite{MR2564571} are only implemented in the case $\mathcal{A}=\Z[G]$.
If $L/\Q$ is at most tamely ramified then it is well-known that $\mathcal{A}=\Z[G]$ and 
$\mathcal{O}_{L}$ is locally free over $\Z[G]$. The above check is therefore implemented 
in this setting with $G=Q_{8}, Q_{12}, Q_{16}$ or $Q_{20}$, in which case
$\Z[G]$ has locally free cancellation by \cite[Theorem I]{MR703486} 
(also see \cite[p.327, (1)]{MR892316}).

In \cite{MR1313877}, Cougnard gives an example of a tamely ramified $Q_{32}$-extension $L/\Q$ for which $\mathcal{O}_{L}$ is stably free but not free over $\Z[Q_{32}]$. 
In this case (H2\textprime) is satisfied (see Proposition \ref{prop:h2prime-holds}(iv)), 
but unfortunately the extension is too large for our implementation to verify
in a reasonable amount of time that a generator does not exist.

In \cite{MR1827291}, examples are given of tamely ramified $Q_{8} \times C_{2}$-extensions $L/\Q$ for which $\mathcal{O}_{L}$ is stably free but not free over $\Z[Q_{8} \times C_{2}]$ 
(this is the smallest group for which the cancellation property fails - see \cite[p.327]{MR892316}). 
In the sample file we applied our algorithm to Cougnard's examples.
This provides an excellent check for the validity of our implementation (for details see the sample file).

\section{Acknowledgments}\label{acknowledgments}

The authors would like to thank: Claus Fieker for his help regarding optimisation of certain Magma commands and for useful discussions regarding \cite{MR2531221}; J\"urgen Kl\"uners for his help in finding polynomials to give generalised quaternion extensions; John Voight for helpful discussions regarding \cite{MR2592031}; and the referee for several helpful comments. 

\bibliography{CompGensII_Bib}{}
\bibliographystyle{amsalpha}

\end{document}